\documentclass[12pt, preprint]{amsart}

\usepackage[top=0.9in, bottom=1in, left=1in, right=1in]{geometry}
\usepackage{amsmath, amssymb, amsthm}
\usepackage{verbatim}
\usepackage{graphicx}
\usepackage{enumerate}
\usepackage{mathrsfs}
\usepackage{tcolorbox}
\usepackage{hyperref}
\usepackage{color}
\usepackage{tikz-cd}
\usepackage{csquotes}
\usepackage{hyperref}
\usepackage{bbm}

\definecolor{darkgreen}{rgb}{0.4,0.0,0.0}

\newtheorem{thm}{Theorem}[section]
\newtheorem{prop}[thm]{Proposition}
\newtheorem{lem}[thm]{Lemma}
\newtheorem{cor}[thm]{Corollary}

\newcommand{\fA}{\mathfrak A}

     \newcommand{\sF}{\mathcal F}

\newcommand{\fM}{\mathfrak M}     
\newcommand{\fN}{\mathfrak N}     \newcommand{\sN}{\mathcal N}
     \newcommand{\sO}{\mathcal O}

     \newcommand{\sR}{\mathcal R}
     
     \newcommand{\sT}{\mathcal T}
     
     \newcommand{\sV}{\mathcal V}

\def\XXint#1#2#3{{\setbox0=\hbox{$#1{#2#3}{\int}$ }
\vcenter{\hbox{$#2#3$ }}\kern-.6\wd0}}

\newcommand{\R}{\mathbb{R}}

\newcommand{\C}{\mathbb{C}}

\newcommand{\N}{\mathbb{N}}

\theoremstyle{definition}
\newtheorem{definition}[thm]{Definition}
\newtheorem{example}[thm]{Example}

\newcommand{\afr}{\mathscr R _{\textrm{aff}}}
\newcommand{\afrprod}{\, \hat \cdot \,}
\newcommand{\afrsum}{\, \hat + \,}
\newcommand{\afrdiff}{\, \hat -\,}

\newcommand\restr[2]{{
  \left.\kern-\nulldelimiterspace 
  #1 
  \vphantom{\big|} 
  \right|_{#2} 
  }}

\theoremstyle{remark}

\newtheorem{remark}[thm]{Remark}

\numberwithin{equation}{section}

\newtheoremstyle{ser}
{8pt}
{8pt}
{\it}
{}
{\sf}
{:}
{6mm}
{}

\newtheoremstyle{serr}
{8pt}
{8pt}
{\normalfont}
{}
{\sf}
{.}
{6mm}
{}

\theoremstyle{ser}
\newtheorem{claim}{Claim}

\theoremstyle{serr}
\newtheorem{claimpff}{Proof of Claim}

\newcommand\Wtilde[1]{\stackrel{\sim}{\smash{\mathscr{#1}}\hspace{0.12in}\rule[0in]{0pt}{1.15ex}}\hspace{-0.12in}}

\title[On Murray-von Neumann Algebras -  I]{On Murray-von Neumann Algebras -  I: Topological, Order-theoretic and Analytical Aspects}
\author{Soumyashant Nayak}

\begin{document}
\address{Indian Statistical Institute\\
 Statistics and Mathematics Unit\\
 8th Mile, Mysore Road\\
 Bengaluru -- 560 059, Karnataka, India.
   ORCiD: 0000-0002-6643-6574}
\email{soumyashant@gmail.com}
\urladdr{https://nsoum.github.io/}

\maketitle

\begin{abstract}
For a countably decomposable finite von Neumann algebra $\mathscr{R}$, we show that any choice of a faithful normal tracial state on $\mathscr{R}$ engenders the same measure topology on $\mathscr{R}$ in the sense of Nelson (J.\ Func.\ Anal., 15 (1974), 103--116). Consequently it is justified to speak of `the' measure topology of $\mathscr{R}$. Having made this observation, we extend the notion of measure topology to general finite von Neumann algebras and denominate it the $\mathfrak{m}$-topology.\ We note that the procedure of $\mathfrak{m}$-completion yields Murray-von Neumann algebras in a functorial manner and provides them with an intrinsic description as unital ordered complex topological $*$-algebras. This enables the study of abstract Murray-von Neumann algebras avoiding reference to a Hilbert space. Furthermore, it makes apparent the appropriate notion of Murray-von Neumann subalgebras, and the intrinsic nature of the spectrum and point spectrum of elements, independent of their ambient Murray-von Neumann algebra. In this context, we show the well-definedness of the Borel function calculus for normal elements and use it along with approximation techniques in the $\mathfrak{m}$-topology to transfer many standard operator inequalities involving bounded self-adjoint operators to the setting of (unbounded) self-adjoint operators in Murray-von Neumann algebras.  On the algebraic side, Murray-von Neumann algebras have been described as the Ore localization of finite von Neumann algebras with respect to their corresponding multiplicative subset of non-zero-divisors. Our discussion reveals that, in addition, there are fundamental topological, order-theoretic and analytical facets to their description.

\vskip 0.01in \noindent
{\bf MSC2010 subject classification:} 47L60, 46L51, 46L10
\end{abstract}

\section{Introduction}
\label{sec:intro}

Finite von Neumann algebras are von Neumann algebras in which every isometry is a unitary. Finite-dimensional complex matrix algebras, abelian von Neumann algebras, group von Neumann algebras for discrete groups, etc., provide a rich assortment of examples of finite von Neumann algebras, helping realize fundamental connections to various parts of mathematics such as probability (free probability), geometry ($L^2$-invariants), non-commutative analysis, etc. They also play a central role in the classification theory of von Neumann algebras. In this article, we are interested in the algebras of unbounded operators known as Murray-von Neumann algebras, which are intimately connected with finite von Neumann algebras. 

 Let $\mathscr{R}$ be a finite von Neumann algebra acting on the Hilbert space $\mathscr{H}$. We say that a closed densely-defined operator on $\mathscr{H}$ is {\it affiliated} with $\mathscr{R}$ if for every unitary $U$ in the commutant of $\mathscr{R}$, we have $U^*TU = T$, where it is understood that $U$ leaves the domain of $T$ invariant. The set of closed densely-defined operators affiliated with $\mathscr{R}$, which we denote by $\afr$, may be naturally endowed with the structure of a unital $*$-algebra (see \cite[Theorem XV]{rings-mvn}, \cite[\S 6.2]{kadison-liu}) and with this algebraic structure, $\afr$ is called the Murray-von Neumann algebra associated with $\mathscr{R}$. A common theme in the theory of operator algebras is the study of the nexus between the spatial theory (action on Hilbert space), and the abstract algebraic theory. The definition of $\afr$ we just introduced is very much from the spatial side, involving the commutant of $\mathscr{R}$ and affiliated operators. This article may be regarded as a quest to identify the appropriate intrinsic definition of Murray-von Neumann algebras, so as to open up the possibility of studying {\it abstract Murray-von Neumann algebras} similar to how one studies abstract $C^*$-algebras or abstract von Neumann algebras avoiding reference to a Hilbert space. An immediate advantage of an abstract approach is that it would allow us to bypass tricky arguments which involve `domain tracking' for unbounded operators. Although there are some hints in this direction in the literature, a reasonably complete picture taking into account the topological, order-theoretic and analytical aspects does not seem to be available. 
 
 In \cite{roos}, \cite{berberian}, $\afr$ is described as the ``maximal quotient ring" of $\mathscr{R}$. In \cite{handelman}, it is established that $\afr$ is, in fact, the Ore localization of $\mathscr{R}$ with respect to its multiplicative subset of non-zero-divisors. In other words, any non-zero-divisor in $\mathscr{R}$ has an inverse in $\afr$. Thus $\afr$ has several useful intrinsic algebraic properties (such as von Neumann regularity) rendering it a powerful object in many contexts. For example, in the proof of a version of the Atiyah conjecture discussed in \cite[Chapter 10]{luck}, it serves as a home for many algebras associated with the group of interest in the context of the corresponding group von Neumann algebra. In this article, we strive to look beyond the purely algebraic aspects of Murray-von Neumann algebras. 

For the discussion in this paragraph and the next, we assume that $\mathscr{R}$ is countably decomposable and thus possesses a faithful normal tracial state. The $*$-algebraic structure of $\afr$ may also be deduced from Nelson's theory of non-commutative integration by viewing $\afr$ as the completion of $\mathscr{R}$ in the $\tau$-measure topology (see \cite[Theorem(s) 1-4]{nelson}).\footnote{Although the theory is developed in the more general setting of semifinite von Neumann algebras with a faithful normal semifinite trace, we are primarily concerned with the case of finite von Neumann algebras since our interest is in Murray-von Neumann algebras.} With this in mind, the temptation is great to hypothesize the existence of an intrinsic topological structure on Murray-von Neumann algebras with an obvious candidate being some flavor of measure topology. This is in contrast with the perception in some of the literature that $\afr$ has no natural topology. As a case in point, in \cite[Chapter 8, pg. 317]{luck}, the following remark is made about $\afr$, ``It does not come with a natural topology anymore but has nice ring theoretic properties." (Also see \cite[pg. 304]{reich}). Our interest is piqued by these conflicting impressions and in the next paragraph, we indulge in some speculative rationalization of this discordance.

From the remark following \cite[Theorem 4]{nelson}, the completion of $\mathscr{R}$ in the $\tau$-measure topology results in a unital $*$-algebra which is $*$-isomorphic to $\afr$ (with the isomorphism 
extending the identity mapping on $\mathscr{R}$), and hence does not depend on the choice of $\tau$. Even
so, in \cite{nelson} this observation is not put into the context of the results by Murray and von 
Neumann (\cite[Theorem XV]{rings-mvn}). Although this seminal work of Nelson's, inspired by Segal's work 
on non-commutative abstract integration (cf.\ \cite{segal}, \cite{segal-correction}), has since been 
extensively used in the theory of noncommutative $L^p$-spaces, tracial inequalities, etc.\ (in the 
semifinite setting), many fundamental questions about the $\tau$-measure topology do not seem to have 
been explored in \cite{nelson} or in subsequent work in this area. For instance, does the $\tau$-measure 
topology (say, in the finite von Neumann algebra setting) depend on the choice of the faithful normal finite trace $\tau$? Firstly, for a 
non-factorial finite von Neumann algebra that is countably decomposable, note that the choice of a faithful normal tracial state is not unique. Secondly, it need not even be `essentially' unique in the following sense; 
given faithful normal tracial states, $\tau _1$ and $\tau _2$, there need not exist positive real numbers $a, b > 0$ satisfying $a \tau_1(A) \le  \tau_2(A) \le b \tau_1(A)$ for all positive operators $A \in \mathscr{R}$.\footnote{For example, let $\tau _1, \tau _2$  be normal states on $L^{\infty}(\R; \mu)$ ($\mu$ being the Lebesgue measure) corresponding to integration with respect to two distinct Gaussian probability measures.} Thus, {\it a priori}, we may be inclined to believe that the $\tau$-measure topology depends on the choice of $\tau$. 

In Theorem \ref{thm:key_thm}, using a non-commutative analogue of the notion of absolute continuity (see 
Lemma \ref{lem:absolute_continuity}) we show that, on the contrary, any choice of a faithful normal 
tracial state on $\mathscr{R}$ engenders the same measure topology on $\mathscr{R}$.  This motivates us 
to appropriately extend the definition of measure topology to general finite von Neumann algebras (that 
are not necessarily countably decomposable) which we denominate the $\mathfrak{m}$-topology. The reader 
is encouraged to think of $\mathfrak{m}$ as `measure' or `measure-theoretic'. We then view Murray-von 
Neumann algebras as completions of finite von Neumann algebras in their respective $\mathfrak{m}$-topologies (see Theorem \ref{thm:aff_op_complete}).

In the remaining part of this section, we directly quote some results (with the numbering that appears in the article) and discuss their significance to better explain the rationale and overall flow of our work.
 \vskip 0.1in
\noindent {\bf Theorem \ref{thm:functorial}}
\textsl{
Let $\mathscr{R}_1$ and $\mathscr{R}_2$ be finite von Neumann algebras and $\Phi : \mathscr{R}_1 \to \mathscr{R}_2$ be a unital $*$-homomorphism. Then the following are equivalent:
\begin{itemize}
    \item[(i)] $\Phi$ is normal;
    \item[(ii)] $\Phi$ is Cauchy-continuous\footnote{A linear map between topological vector spaces is Cauchy-continuous if and only if it is continuous.} for the $\mathfrak{m}$-topologies on $\mathscr{R}_1$ and $\mathscr{R}_2$.
\end{itemize}
}
\noindent It follows from Theorem \ref{thm:functorial} that completion in the $\mathfrak{m}$-topology, or $\mathfrak{m}$-completion, is a functor from the category of finite von Neumann algebras with morphisms as normal unital $*$-homomorphisms, to the category of unital complex topological $*$-algebras with morphisms as continuous unital $*$-homomorphisms. In \cite{nayak-matrix}, the measure topology (or $\mathfrak{m}$-topology) has been investigated by the present author in the setting of $II_1$ factors where matters are more straightforward owing to the existence of a {\it unique} faithful normal tracial state. Having understood the intrinsic nature of the $\mathfrak{m}$-topology in the context of finite von Neumann algebras, many of the results and their proofs in \cite{nayak-matrix} can be transferred to the general setting almost verbatim. Let $\Wtilde{R}$ denote the $\mathfrak{m}$-completion of $\mathscr{R}$, and $\mathscr{R}^{+}$ denote the cone of positive operators in the space of self-adjoint elements in $\mathscr{R}$. In Proposition \ref{prop:mtilde_pos}, we observe that the $\mathfrak{m}$-closure of $\mathscr{R}^{+}$ in $\Wtilde{R}$ yields a cone for $\Wtilde{R} ^{\textrm{sa}}$, the space of self-adjoint elements in $\Wtilde{R}$. Thus $\Wtilde{R} ^{\!\textrm{sa}}$ naturally has an order structure with respect to the above cone. This intrinsic order structure manifests as the usual order structure obtained from the cone of positive affiliated operators in $\afr$. With this in mind, we think of Murray-von Neumann algebras as unital ordered complex topological $*$-algebras.\footnote{By an ordered complex $*$-algebra, we mean a complex $*$-algebra whose Hermitian elements form an ordered real vector space.} 

By virtue of Theorem \ref{thm:functorial}, the image of the $\mathfrak{m}$-completion functor is a subcategory of the category of unital topological $*$-algebras with objects as Murray-von Neumann algebras and morphisms as the $\mathfrak{m}$-continuous unital $*$-homomorphisms.  Often in operator algebras, the topological structure and the order structure are strongly intertwined. Let $\fM$  be a Murray-von Neumann algebra and let $\fM ^{\textrm{sa}}$ denote the set of self-adjoint elements of $\fM$. In Proposition \ref{prop:lub_mvna}, we show that any monotonically increasing net of self-adjoint elements in $\fM$ that is bounded above (by an element of $\fM ^{\textrm{sa}}$), has a least upper bound in $\fM ^{\textrm{sa}}$. In other words, $\fM^{\textrm{sa}}$ is {\it monotone complete}.

For Murray-von Neumann algebras $\fM _1$ and $\fM _2$, we say that a unital $*$-homomorphism $\Phi : \fM _1 \to \fM _2$ is {\it normal} if for every monotonically increasing net $\{ H_i \}$ of self-adjoint elements in $\fM _1$ which has an upper bound in $\fM _1$, we have $\Phi(\sup H_i) = \sup \Phi(H_i)$ in $\fM _2$.
\vskip 0.1in
\noindent {\bf Theorem \ref{thm:normal_mvna}}
\textsl{
Let $\fM _1$ and $\fM _2$ be Murray-von Neumann algebras and $\Phi : \fM _1 \to \fM _2$ be a unital $*$-homomorphism. Then the following are equivalent:
\begin{itemize}
    \item[(i)] $\Phi$ is normal;
    \item[(ii)] $\Phi$ is $\mathfrak{m}$-continuous.
\end{itemize}
}
\noindent Thus the normal unital $*$-homomorphisms may be equivalently considered to be the morphisms in the category of Murray-von Neumann algebras.

 For a clearer picture, we briefly explain the interplay between the algebraic, topological and order-theoretic aspects of Murray-von Neumann algebras.  Since the positive cone of a Murray-von Neumann algebra can be algebraically described as the set of squares of self-adjoint elements (see Proposition \ref{prop:pos_desc_mvn}), any unital $*$-homomorphism between Murray-von Neumann algebras is automatically order-preserving. Note that the underlying finite von Neumann algebra of a Murray-von Neumann algebra $\fM$ may be extracted with the help of the order structure as follows, $$\mathscr{R} = \{ A \in \fM : \exists \lambda \in \R_{+} \textrm{ with }  A^*A \le \lambda ^2 I \},$$
and the norm of an element $A \in \mathscr{R}$ is given by $$\|A\| := \inf \; \{ \lambda \in \R_{+} : A^*A \le \lambda^2 I \}.$$ Thus a unital $*$-homomorphism between Murray-von Neumann algebras must send norm-bounded elements to norm-bounded elements, inducing a unital $*$-homomorphism between the underlying finite von Neumann algebras. In light of Theorem \ref{thm:normal_mvna} and Theorem \ref{thm:functorial}, we see that every morphism between Murray-von Neumann algebras arises from a morphism between their underlying finite von Neumann algebras.

The functoriality of the $\mathfrak{m}$-completion of finite von Neumann algebras is eminently useful in identifying and clarifying the intrinsic nature of many notions that arise in the context of Murray-von Neumann algebras. For instance, {\it Murray-von Neumann subalgebras} of $\afr$ may be defined as $\mathfrak{m}$-closed, or simply, closed $*$-subalgebras of $\afr$ containing the same identity as $\afr$. Note that every von Neumann subalgebra of a finite von Neumann algebra is finite. 
\vskip 0.1in
\noindent {\bf Proposition \ref{prop:compatibility_subalgebra}}
\textsl{
Let $\mathscr{R}$ be a finite von Neumann algebra and $\mathscr{S}$ be a von Neumann subalgebra of $\mathscr{R}$. Then the $\mathfrak{m}$-topology on $\mathscr{S}$ is identical to the subspace topology induced from the $\mathfrak{m}$-topology on $\mathscr{R}$. 
}
\vskip 0.1in
\noindent From Proposition \ref{prop:compatibility_subalgebra}, we see that if $\Wtilde{S}$ is a Murray-von Neumann subalgebra of $\afr$, then $\mathscr{S} := \;\Wtilde{S} \cap \; \mathscr{R}$ is a von Neumann subalgebra of $\mathscr{R}$ and $\Wtilde{S}\; \cong \mathscr{S}_{\textrm{aff}}$. In other words, Murray-von Neumann subalgebras of $\afr$ naturally arise as the $\mathfrak{m}$-completion of von Neumann subalgebras of $\mathscr{R}$.

In \S \ref{sec:abelian_mvn}, we characterize the abelian Murray-von Neumann algebras. Let $X$ be a 
locally compact Hausdorff space with a positive Radon measure $\mu$. The space $L^0(X;\mu)$ of
$\mu$-measurable functions on $X$ equipped with the topology of local convergence in measure is a 
complete topological $*$-algebra (see Proposition \ref{prop:complete_L0}). The set of essentially positive $\mu$-measurable functions on $X$ is a 
positive cone for $L^0(X; \mu)$. With the above-mentioned topological and order structures, $L^0(X; \mu)$ is an abelian Murray-von Neumann algebra and its underlying finite von Neumann algebra is given by
$L^{\infty}(X; \mu)$, the space of essentially bounded $\mu$-measurable functions on $X$. In fact, all abelian Murray-von Neumann algebras arise in this manner.

In \S \ref{sec:approx_meas}, we algebraically define the notions of {\it spectrum} and {\it point 
spectrum} for an element of an abstract Murray-von Neumann algebra (see Definition \ref{def:spectrum}) and note that they are independent of 
the ambient Murray-von Neumann algebra (see Proposition \ref{prop:spec_inv}). More specifically, for Murray-von Neumann algebras $\fM, \fN$ 
with $\fN \hookrightarrow \fM$ and an operator $A \in \fN$, the spectrum of $A$ 
relative to $\fN$ is identical to the spectrum of $A$ relative to $\fM$. The same goes for the point spectrum. Hence we may use the notation $\mathrm{sp}(A)$ ($\mathrm{sp}_e (A)$, respectively) to denote the spectrum of $A$ (point spectrum of $A$, respectively) without reference to its ambient Murray-von Neumann algebra.\footnote{For operators in $\afr$ (acting on the Hilbert space $\mathscr{H}$), this coincides with the usual 
definition of spectrum and point spectrum.} Furthermore, we note that the spectrum of an element of a 
Murray-von Neumann algebra is a closed subset of $\C$, and for a normal element, the spectrum is 
non-empty. Thus the Borel function calculus of a normal element $A \in \fM$ is well-defined as the unique $\sigma$-normal $*$-homomorphism from $\mathscr{B}_u(\mathrm{sp}(A))$, the space of Borel functions on 
$\mathrm{sp}(A)$, into $\fM$, mapping the constant function $1$ onto $I$ and the identity transformation 
$\iota$ on $\mathrm{sp}(A)$ onto $A$.

For $n \in \N$, note that $M_n(\mathscr{R})$ is a finite von Neumann algebra. From Theorem \ref{thm:matrix-iso}, we have the isomorphism $M_n(\afr) \cong M_n(\mathscr{R})_{\textrm{aff}}$, which provides the foundations for matrix-theoretic arguments in Murray-von Neumann algebras. With the help of the Borel function calculus, approximation techniques in the $\mathfrak{m}$-topology, and matrix-theoretic arguments, we transfer several fundamental operator  inequalities (see Theorem \ref{thm:op-mon}, \ref{thm:op-conv}) involving bounded self-adjoint operators to the setting of (unbounded) self-adjoint operators in a Murray-von Neumann algebra.
\vskip 0.1in
\noindent {\bf Theorem \ref{thm:op-mon}}
\textsl{
Let $f: [0, \infty) \to \R$ be an operator monotone function and $A, B \in \afr$ be positive operators such that $A \le B$. Then $f(A) \le f(B)$.
}
\vskip 0.1in

\noindent {\bf Theorem \ref{thm:op-conv}}
\textsl{
For an operator convex function $f: [0, \infty) \to \R$ with $f(0) = 0$, we have the following:
\begin{itemize}
    \item[(i)] $f(tA \afrsum (1-t)B) \le tf(A) \afrsum (1-t)f(B)$ for all positive operators $A, B$ in $\afr$;\footnote{For $A, B \in \afr$, $A \afrsum B := \overline{A+B}, A \afrprod B := \overline{AB}$.}
    \item[(ii)] $f(V^* \afrprod A \afrprod V) \le V^* \afrprod f(A) \afrprod V$ for every positive contraction $V$ in $\mathscr{R}$ (that is, $\|V\| \le 1$) and every positive operator $A$ in $\afr$;
    \item[(iii)] $f(V^* \afrprod A \afrprod V \afrsum  W^* \afrprod B \afrprod W) \le V^* \afrprod f(A) \afrprod V \afrsum  W^* \afrprod f(B) \afrprod W$ for all $V, W$ in $\mathscr{R}$ with $V^*V + W^*W \le I$ and all positive operators $A, B$ in $\afr$;
    \item[(iv)] $f(E \afrprod A \afrprod E) \le E \afrprod f(A) \afrprod E$ for every projection $E$ in $\mathscr{R}$ and every positive operator $A$ in $\afr$.
\end{itemize} 
}

\subsection{Acknowledgments}
I would like to thank Dmitri Pavlov for his illuminating answers on the website \href{https://mathoverflow.net}{MathOverflow} (in particular, \cite{pavlov}) and sharing his study of the opposite category of commutative von Neumann algebras in \cite{pavlov-gelfand}. I am also grateful to Zhe Liu and Raghavendra Venkatraman for helpful discussions, and to K. V. Shuddhodan for  valuable feedback that helped improve the exposition of the paper.

\section{Preliminaries}

\subsection{Notation and Terminology}
Throughout this article, $\mathscr{H}$ denotes a Hilbert space over the complex numbers (usually infinite-dimensional, though not necessarily separable), $\mathscr{R}$ denotes a finite von Neumann algebra acting on $\mathscr{H}$, and $\fM$ denotes a Murray-von Neumann algebra. For an unbounded operator $T$ acting on $\mathscr{H}$, we denote its domain of definition by $\mathscr{D}(T)$. We use the term `unbounded' to mean `not necessarily bounded' rather than `not bounded'. The closure of a pre-closed operator $T$ is denoted by $\overline{T}$.

A complex $*$-algebra $\mathfrak{A}$ is said to be {\it ordered} if the Hermitian elements in $\mathfrak{A}$ form an ordered real vector space. For an ordered complex $*$-algebra $\mathfrak{A}$ (such as von Neumann algebras, Murray-von Neumann algebras, etc.), we denote the set of self-adjoint elements in $\mathfrak{A}$ by $\mathfrak{A}^{\mathrm{sa}}$, and the positive cone of $\mathfrak{A}^{\mathrm{sa}}$ by $\mathfrak{A}^{+}$. Any remark on the order properties of $\fA$ actually pertain to $\mathfrak{A}^{\mathrm{sa}}$. (For example, when we say $\fA$ is monotone complete, we really mean $\mathfrak{A}^{\mathrm{sa}}$ is monotone complete.)

 The identity operator in $\mathfrak{A}$ is denoted by $I$ and the identity matrix of $M_n(\mathfrak{A})$ is denoted by ${\bf I}_n$. We denote a net in $\fA$ by $\{ T_i \}$ suppressing the indexing set of $i$ (denoted by $\Lambda$) when it is clear from the context. A function (or $\mu$-measurable function for a measure $\mu$) on a topological space is complex-valued unless stated otherwise. For a subset $S$ of $\C$, the set of bounded Borel functions on $S$ is denoted by $\mathscr{B}_b(S)$, and the set of Borel functions on $S$ is denoted by $\mathscr{B}_u(S)$. The general references used are \cite{kadison-ringrose1}, \cite{kadison-ringrose2}, \cite{takesaki1}.
 
A partially ordered vector space $V$ is said to be {\it monotone $\sigma$-complete} if every monotone increasing sequence of elements in $V$ that is bounded above, has a least upper bound in $V$. More generally, if every monotone increasing net in $V$ that is bounded above has a least upper bound in $V$, then $V$ is said to be {\it monotone complete}. 
 
 \begin{definition}
Let $V_1,  V_2$ be monotone-$\sigma$-complete partially ordered vector spaces and $\Phi : V_1 \to V_2$ be an order-preserving linear map. For a monotone increasing sequence $\{ H_n \}_{n \in \N}$ in $V_1$ that is bounded above, clearly $\{ \Phi(H_n)\}_{n \in \N}$ is a monotone increasing sequence in $V_2$ that is bounded above. We say that the map $\Phi$ is {\it $\sigma$-normal} if for any such sequence $\{H_n\}$ in $V_1$, we have $\Phi(\sup H_n) = \sup \Phi(H_n)$.

Similarly, let $V_1,  V_2$ be monotone-complete partially ordered vector spaces and $\Phi : V_1 \to V_2$ be an order-preserving linear map. For a monotone increasing net $\{ H_i \}$ in $V_1$ that is bounded above, $\{ \Phi(H_i)\}$ is a monotone increasing sequence in $V_2$ that is bounded above. We say that the map $\Phi$ is {\it normal} if for any such net $\{H_i\}$ in $V_1$, we have $\Phi(\sup H_i) = \sup \Phi(H_i)$.
\end{definition}

\begin{example}
The self-adjoint part of a von Neumann algebra is monotone complete with respect to the usual operator ordering. The space of real-valued bounded Borel functions on $S \subseteq \C$ is monotone $\sigma$-complete with respect to the ordering corresponding to the cone of positive functions therein. Similarly, the space of real-valued unbounded Borel functions on $S \subseteq \C$ is monotone $\sigma$-complete.
\end{example}

\subsection{Relevant concepts from the theory of von Neumann algebras}
In this subsection, we recall some basic definitions and results from the theory of von Neumann algebras that are relevant to our discussion. We also prove non-commutative analogues of two results from basic measure theory (see Lemma \ref{lem:absolute_continuity}, Lemma \ref{lem:lebesgue-radon-nikodym}) that are crucial to our discussion.

Let $\mathscr{M}$ be a von Neumann algebra acting on the Hilbert space $\mathscr{H}$. The set of projections in $\mathscr{M}$, partially ordered by the usual operator ordering, is a complete lattice. In other words, every family of projections $\{ E_i \}$ in $\mathscr{M}$ has a least upper bound, and a greatest lower bound, in $\mathscr{M}$. We denote the least upper bound by $\vee _i E_i$ and call it the {\it union} of the family $\{ E_i \}$. We denote the greatest lower bound by $\wedge _i E_i$ and call it the {\it intersection} of the family $\{E_i\}$.

Two projections $E$ and $F$ in $\mathscr{M}$ are said to be {\it equivalent} (written $E \sim F$) when $V^*V = E$ and $VV^* = F$ for some $V \in \mathscr{M}$. We say that $E$ is {\it weaker} than $F$ (written $E \lesssim F$) when $E$ is equivalent to a subprojection of $F$. The reader may refer to \cite[Chapter 6]{kadison-ringrose2} for a detailed account of the comparison theory of projections.

\begin{definition}
 For an operator $T$ in $\mathscr{M}$, the projection onto the closure of the range of $T$ in $\mathscr{H}$ is said to be the {\it range projection} of $T$, and denoted by $\mathcal{R}(T)$. The projection onto the null space of $T$ is denoted by $\mathcal{N}(T)$. 

The range projection of $T$ is the smallest projection in $\mathscr{M}$ amongst all projections $E$ in $\mathscr{M}$ satisfying $ET = T$. Similarly, $\sN (T)$ is the largest projection $F$ in $\mathscr{M}$ such that $TF = 0$.
\end{definition}

\begin{prop}[see {\cite[Proposition 2.5.13]{kadison-ringrose1}, \cite[Proposition 6.1.6]{kadison-ringrose2}}]
\label{prop:range_proj}
\textsl{
Let $\mathscr{M}$ be a von Neumann algebra acting on the Hilbert space $\mathscr{H}$. For an operator $T$ in $\mathscr{M}$, we have:
\begin{itemize}
\item[(i)] $\mathcal{R}(T) = I - \mathcal{N}(T^*)$ and $\mathcal{N}(T) = I - \mathcal{R}(T^*)$;
\item[(ii)] $\mathcal{R}(T) = \mathcal{R}(TT^*)$ and $\mathcal{N}(T) = \mathcal{N}(T^*T)$;
\item[(iii)] $\mathcal{R}(T)$ and $\mathcal{N}(T)$ are in $\mathscr{M}$;
\item[(iv)] $\mathcal{R}(T) \sim \mathcal{R}(T^*)$ relative to $\mathscr{M}$.
\end{itemize}
}
\end{prop}

\begin{remark}
A normal (or $\sigma$-normal) linear map between ordered $*$-algebras is automatically $*$-preserving as by definition, it is order-preserving. In the case of normal (or $\sigma$-normal) homomorphisms, for emphasis we prefer to say normal (or $\sigma$-normal) $*$-homomorphism.
\end{remark}

\begin{definition}
A {\it normal state} of a von Neumann algebra $\mathscr{M}$ is a positive linear functional $\omega$ of $\mathscr{M}$ such that $\omega(I) = 1$ and for any increasing net of projections $\{ E_i \}$ in $\mathscr{M}$, we have $\omega(\sup E_i) = \sup \omega(E_i)$. The {\it support projection} of a normal state $\omega$ of $\mathscr{M}$ is the (unique) smallest projection $E$ in $\mathscr{M}$ such that $\omega(I-E) = 0$. We denote the support projection of $\omega$ by $S_{\omega}$. In fact, $\omega(TS_{\omega}) = \omega(S_{\omega}T) = \omega(T)$ for all operators $T \in \mathscr{M}$.

A normal state $\omega$ of $\mathscr{M}$ is said to be {\it tracial} if $\omega(AB) = \omega(BA)$ for all $A, B \in \mathscr{M}$. For a normal tracial state, the support projection lies in the center of $\mathscr{M}$.
\end{definition}

Note that for a normal state $\omega$ on a von Neumann algebra $\mathscr{M}$ and a decreasing net of projections $\{ E_i \}$ in $\mathscr{M}$, we have $\omega (\inf E_i) = \inf \omega (E_i)$.

\begin{remark}
\label{rmrk:tr_ineq}
Let $\mathscr{M}$ be a von Neumann algebra and $\tau$ be a normal tracial state on $\mathscr{M}$. For projections $E_1, E_2$ in $\mathscr{M}$, using the Kaplansky identity $(E_1 \vee E_2 - E_1) \sim (E_2 - E_1 \wedge E_2)$ and the fact that $\tau$ is tracial, we have $$\tau (E_1 \vee E_2) = \tau(E_1) + \tau (E_2) - \tau (E_1 \wedge E_2)$$ which implies that $\tau (E_1 \vee E_2) \le \tau (E_1) + \tau (E_2)$. By induction, we see that for projections $E_1, E_2, \ldots, E_n$ in $\mathscr{M}$, we have  $$\tau (\bigvee_{i=1}^n E_i) \le \sum_{i=1}^n \tau (E_i).$$ Using the normality of $\tau$, we conclude that for countably many projections $E_1, E_2, \ldots$ in $\mathscr{M}$, we have $$\tau (\bigvee_{i=1}^{\infty} E_i) \le \sum_{i=1}^{\infty} \tau (E_i).$$
\end{remark}

\begin{lem}[Absolute continuity]
\label{lem:absolute_continuity}
\textsl{
Let $\mathscr{M}$ be a von Neumann algebra. Let $\tau _1, \tau _2$ be normal tracial states on $\mathscr{M}$ such that $S_{\tau _2} \le S_{\tau _1}$. Then for every $\delta _2 > 0$, there exists $\delta _1 > 0$ such that whenever $\tau _1(E) < \delta _1$ for a projection $E$ in $\mathscr{M}$, we have $\tau _2 (E) < \delta _2$.
}
\end{lem}
\begin{proof}
 On the contrary, let us assume that there is a $\varepsilon > 0$ such that for every positive integer $n$, there is a projection $E_n$ in $\mathscr{R}$ such that $\tau _1(E_n) < \frac{1}{2^n}$ and $\tau _2(E_n) \ge \varepsilon$. Consider the following projection, $$F := \bigwedge_{n = 1}^{\infty}  \Big( \bigvee_{k=n}^{\infty} E_k \Big).$$

From Remark \ref{rmrk:tr_ineq}, for $n \in \N$, we have $\tau _1 (\bigvee_{k=n}^{\infty} E_k) \le \sum_{k=n}^{\infty}\tau _1(E_k) < \frac{1}{2^{n-1}}$. From the normality of $\tau _1$, we have $\tau _1(F) = 0$ which implies that $S_{\tau _1}F = 0$. Since $S_{\tau _2} \le S_{\tau _1}$, we conclude that $S_{\tau _2}F = 0$. On the other hand, since $\tau _2 (\bigvee_{k=n}^{\infty} E_k) \ge \varepsilon$ for all $n \in \N$, by the normality of $\tau _2$, we have $\tau _2(F) = \tau _2( S_{\tau _2} F) \ge \varepsilon$. This leads us to a contradiction.
\end{proof}

The classical Lebesgue decomposition theorem (see \cite[Theorem 6.10]{rudin2}) states that for every pair of $\sigma$-finite measures $\mu, \nu$ on a measurable space, there exist two $\sigma$-finite signed measures $\nu _1$ and $\nu_2$ on the space such that 
\begin{itemize}
    \item[(i)] $\nu = \nu_1 + \nu_2$;
    \item[(ii)] $\nu_1$  and $\mu$ are mutually singular;
    \item[(iii)] $\nu_2$ is absolutely continuous with respect to $\mu$.
\end{itemize}
We note a non-commutative version of the Lebesgue decomposition theorem below.
\begin{lem}[Non-commutative Lebesgue decomposition]
\label{lem:lebesgue-radon-nikodym}
\textsl{
Let $\mathscr{R}$ be a finite von Neumann algebra and $\tau, \rho$ be normal tracial states on $\mathscr{R}$. Then $\rho$ may be decomposed into a convex combination of two normal tracial states $\rho _1, \rho _2$ on $\mathscr{R}$ satisfying $S_{\rho _1} S_{\tau} = 0,$ and $S_{\rho _2 } \le S_{\tau}$. (In analogy with the classical Lebesgue decomposition theorem, $\rho _1$  and $\tau$ are mutually singular, and $\rho _2$ is absolutely continuous with respect to $\tau$.)
}
\end{lem}
\begin{proof}
 If $S_{\rho} S_{\tau} = 0$, that is, the support projections of $\tau$ and $\rho$ are orthogonal, we may decompose $\rho$ as $\rho = 1 \cdot \rho + 0 \cdot \tau$ (choosing $\rho _1 = \rho, \rho _2 = \tau$). If $S_{\rho} \le S_{\tau}$, we may choose a normal tracial state $\lambda$ on $\mathscr{R}$ such that the support projection of $\lambda$ is orthogonal to the support projection of $\tau$, and decompose $\rho$ as $\rho = 0 \cdot \lambda + 1 \cdot \rho$ (choosing $\rho _1 = \lambda, \rho _2 = \rho$). Thus for the rest of the proof, we may assume that $S_{\tau} S_{\rho} \ne 0, S_{\rho } (I  - S_{\tau}) \ne 0$.
 
 Let $E := S_{\rho} - S_{\tau} \wedge S_{\rho} = S_{\rho}(I-S_{\tau})$ and $F := S_{\rho} S_{\tau}$. Note that $\rho(E) \ne 0, \rho(F) \ne 0,$ and $\rho(E) + \rho(F) = 1$. Furthermore, $ES_{\tau} = 0$ and $F \le S_{\tau}$. Let $\rho _1, \rho _2$, respectively, be the normal tracial states on $\mathscr{R}$ defined by $$\rho _1(A) = \frac{\rho(EA)}{\rho(E)},\;\; \rho _2(A) = \frac{\rho (F A)}{\rho(F)}, \textrm { for } A \in \mathscr{R}.$$
 We may decompose $\rho$ as $\rho = \rho(E) \cdot \rho _1 + \rho(F) \cdot \rho _2$. Since $S_{\rho _1} = E, S_{\rho _2} = F$, we are done.
\end{proof}

\subsection{Murray-von Neumann algebras}

 A concise account of the theory of unbounded operators may be found in \cite[\S 4]{kadison-liu}. For a more thorough account, the interested reader may refer to \S 2.7, \S 5.6 in \cite{kadison-ringrose1}, or Chapter VIII in \cite{simon-reed}. In this subsection, we recollect the definition of Murray-von Neuman algebras and note some basic facts about them.

Let $\mathscr{R}$ be a finite von Neumann algebra acting on the Hilbert space $\mathscr{H}$. We denote the set of closed densely-defined operators affiliated with $\mathscr{R}$ by $\afr$. 

\begin{prop}[see {\cite[Proposition 6.8]{kadison-liu}}]
\label{prop:fund_mva}
\textsl{
Let $\mathscr{R}$ be a finite von Neumann algebra acting on the Hilbert space $\mathscr{H}$. For operators $A, B$ in $\mathscr{R}_{\textrm{aff}}$, we have:
\begin{itemize}
\item[(i)] $A+B$ is densely defined, preclosed and has a unique closed extension $A \afrsum B (:= \overline{A+B})$ in $\mathscr{R}_{\textrm{aff}}$;
\item[(ii)] $AB$ is densely defined, preclosed and has a unique closed extension $A \afrprod B (:= \overline{AB})$ in $\mathscr{R}_{\textrm{aff}}$.
\end{itemize}
}
\end{prop}

\begin{definition}
In view of \cite[Proposition(s) 6.9-6.12]{kadison-liu}, for a finite von Neumann algebra $\mathscr{R}$, the set of affiliated operators $\afr$ may be endowed with the structure of a $*$-algebra with $\afrsum$ as addition and $\afrprod$ as multiplication. With this $*$-algebraic structure, $\mathscr{R}_{\textrm{aff}}$ is called the {\it Murray-von Neumann algebra} associated with $\mathscr{R}$. 

An operator $P$ in $\mathscr{R}_{\textrm{aff}}$ (which is, {\it ipso facto}, closed and densely defined) is said to be {\it positive} if $\langle Px, x \rangle \ge 0$ for all vectors $x \in \mathscr{D}(P)$. The set of positive operators in $\mathscr{R}_{\textrm{aff}}$ is a cone and with this positive cone, $\mathscr{R}_{\textrm{aff}}$ may be viewed as an ordered $*$-algebra. 
\end{definition}

\begin{prop}[see {\cite[Lemma 8.20, Theorem 8.22]{luck}}]
\label{prop:luck}
\textsl{
Let $A$ be an operator in $\afr$. In the context of the ring $\afr$, the following are equivalent:
\begin{itemize}
    \item[(i)] $A$ is not a left zero-divisor;
    \item[(ii)] $A$ is not a zero-divisor;
    \item[(iii)] $A$ is invertible;
    \item[(iv)] $A$ has dense range;
    \item[(v)] $A$ has trivial nullspace.
\end{itemize}
}
\end{prop}

\section{The Measure Topology}

Let $\mathscr{R}$ be a countably decomposable finite von Neumann algebra with a faithful normal tracial state $\tau$. We start our discussion in this section by showing that the $\tau$-measure topology on $\mathscr{R}$ defined by Nelson in \cite[\S 2]{nelson} is, in fact, independent of the choice of $\tau$.

\begin{definition}
For $\varepsilon, \delta >0$, we define
$\sO _{\tau}(\varepsilon, \delta) 
    := \{ A \in \mathscr{R} : $ for some projection $E \textrm{ in } \mathscr{R} \textrm{ with } \tau(I - E)  \le \delta$, we have $\|AE\| \le \varepsilon \}$.
\end{definition}
The {\it $\tau$-measure topology} of $\mathscr{R}$ is defined as the translation-invariant topology generated by the fundamental system of neighborhoods of $0$ given by $\{\sO _{\tau}(\varepsilon, \delta) \}$. 

\begin{thm}
\label{thm:key_thm}
\textsl{
Let $\tau_1, \tau_2$ be faithful normal tracial states on $\mathscr{R}$. Then the $\tau_1$-measure topology and the $\tau_2$-measure topology on $\mathscr{R}$ are identical.
}
\end{thm}
\begin{proof}
Let $\delta_2 > 0$. Since $\tau _1$ and $\tau _2$ are faithful normal states, we observe that the identity projection is the support projection for both $\tau _1$ and $\tau _2$, that is, $S_{\tau _1} = S_{\tau _2} = I$. By Lemma \ref{lem:absolute_continuity}, there exists $\delta_1 > 0$ such that whenever $\tau_1(I-E) < \delta_1$ for a projection $E$ in $\mathscr{R}$, we have $\tau_2(I-E) < \delta_2$. Thus $\sO _{\tau_1}(\varepsilon, \delta_1) \subseteq \sO _{\tau_2}(\varepsilon, \delta_2)$ for all $\varepsilon > 0$ and we conclude that the $\tau_1$-measure topology is finer than the $\tau_2$-measure topology. By a symmetric argument, we conclude that the $\tau_2$-measure topology is finer than the $\tau_1$-measure topology. Thus the $\tau_1$-measure topology and the $\tau_2$-measure topology on $\mathscr{R}$ are identical.
\end{proof}

In view of Theorem \ref{thm:key_thm}, we may refer to the $\tau$-measure topology simply as the {\it measure topology} on $\mathscr{R}$. Furthermore, it suggests the possibility of defining the measure topology in a way which makes its intrinsic nature obvious from the outset and allows for generalization to finite von Neumann algebras that are not necessarily countably decomposable. Our main goal in this section is to accomplish this task. Although our treatment is heavily inspired by and closely mirrors that of \cite[\S 2]{nelson}, we refrain from handwavy allusions to the results therein and work out all the details in our context. In the rest of this section (and the manuscript), $\mathscr{R}$ denotes a finite von Neumann algebra acting on the Hilbert space $\mathscr{H}$.

\begin{definition}
For $\varepsilon, \delta > 0$ and a normal tracial state $\tau$ on $\mathscr{R}$, we define $\sO(\tau, \varepsilon, \delta) := \{ A \in \mathscr{R} : $ for some projection $E \textrm{ in } \mathscr{R} \textrm{ with } \tau(I - E)  \le \delta$, we have $\|AE\| \le \varepsilon \}$.

The translation-invariant topology generated by the fundamental system of neighborhoods of $0$ given by $\{  \sO(\tau, \varepsilon, \delta) \}$ is called the {\it $\mathfrak{m}$-topology}\footnote{We think of $\mathfrak{m}$ as `measure' or `measure-theoretic'.} of $\mathscr{R}$. 
\end{definition}

Thus a net $\{ A_i \}$ in $\mathscr{R}$ converges in the $\mathfrak{m}$-topology (or {\it converges in measure}) to an operator $A$ in $\mathscr{R}$ if and only if for every triple $(\tau, \varepsilon, \delta)$ with $\tau$ a normal tracial state on $\mathscr{R}$ and $\varepsilon, \delta > 0$, there is an index $j$ such that $A_i - A \in \sO(\tau, \varepsilon, \delta)$ for all $i \ge j$. We say that a net $\{ A_i \}$ in $\mathscr{R}$ is {\it Cauchy in measure} if for every triple $(\tau, \varepsilon, \delta)$ with $\tau$ a normal tracial state on $\mathscr{R}$ and $\varepsilon, \delta > 0$, there is an index $k$ such that $A_i - A_j \in \sO(\tau, \varepsilon, \delta)$ for all $i, j \ge k$ (or equivalently, if there is an index $k$ such that $A_i - A_k \in \sO(\tau, \varepsilon, \delta)$ for all $i \ge k$).

\begin{remark}
\label{rmrk:trace_contain}
Let $\tau_1, \tau_2$ be normal tracial states on $\mathscr{R}$ such that $S_{\tau_2} \le S_{\tau _1}$. By Lemma \ref{lem:absolute_continuity}, for every $\delta _2 > 0$, there exists $\delta _1 > 0$ such that whenever $\tau _1(I-E) < \delta _1$ for a projection $E$ in $\mathscr{R}$, we have $\tau_2(I-E) < \delta _2$. Rephrased in terms of $\mathcal{O}$-neighbourhoods, we see that for every $\delta _2 > 0$, there exists $\delta _1 > 0$ such that $\sO(\tau_1, \varepsilon, \delta _1) \subseteq \sO(\tau_2, \varepsilon, \delta _2)$ for all $\varepsilon > 0$.
\end{remark}

\begin{definition}
For $\varepsilon, \delta > 0$ and a normal tracial state $\tau$ on $\mathscr{R}$, we define $\sV(\tau, \varepsilon, \delta) 
    := \{ x \in \mathscr{H} : $ for some projection $E \textrm{ in } \mathscr{R} \textrm{ with } \tau(I - E)  \le \delta$, we have $\|Ex\| \le \varepsilon \}$.
    
The translation-invariant topology generated by the fundamental system of neighborhoods of $0$ given by $\{ \sV(\tau, \varepsilon, \delta) \}$ is called the {\it $\mathfrak{m}$-topology} of $\mathscr{H}$. 
\end{definition}

For a net of vectors in $\mathscr{H}$, the notions of {\it convergence in measure} and being {\it Cauchy in measure} are defined analogous to the case of operators in $\mathscr{R}$.

\begin{definition}
Let $\sT$ be a separating family of normal tracial states on $\mathscr{R}$.  The topology on $\mathscr{R}$ generated by the fundamental system of neighborhoods of $0$ given by $\{ \sO(\tau, \varepsilon, \delta):\varepsilon, \delta > 0, \tau \in \sT \}$ is called the {\it $\sT$-measure topology} of $\mathscr{R}$.
\end{definition}

\begin{lem}
\label{lem:family_sep_states}
\textsl{
Let $\sT$ be a separating family of normal tracial states on $\mathscr{R}$. Then the $\sT$-measure topology on $\mathscr{R}$ is identical to the $\mathfrak{m}$-topology on $\mathscr{R}$.
}
\end{lem}
\begin{proof}
Let $\tau$ be a normal tracial state on $\mathscr{R}$. 
Let $\sF := \{ \bigvee _{n \in \N} S_{\tau _n} : \tau_n \in \sT \textrm{ for } n \in \N   \}$, that is, $\sF$ consists of countable unions of support projections of the states in $\sT$. Clearly $\sF$ is closed under countable unions of its projections. 

\begin{claim}
\label{claim:family_sep_states1}
\textsl{
$S_ {\tau} \le F$ for some projection $F$  in $\sF $.
}
\end{claim}
\begin{claimpff}
Let $\Gamma$ be the subset of $[0, 1]$ given by $\{ \tau(S_{\tau} \wedge E) : E \in \sF \}$ and $\alpha := \sup \Gamma$. Since $\sT$ is a separating family of states on $\mathscr{R}$, clearly $\alpha > 0$. For $n \in \N$, there is a projection $E_n$ in $\sF$ such that $\tau(S_{\tau} \wedge E_n) \ge \alpha - \frac{1}{n}$. Let $F$ denote the projection $\bigvee _{n \in \N} E_n$ in $\mathcal{F}$ so that $F \ge E_n$ for all $n \in \N$. We have $$\tau(S_{\tau} \wedge F) \ge \sup_{n \in \N} \tau(S_{\tau} \wedge E_n) \ge \alpha .$$
Since $F \in \sF$, we see that $\tau(S_{\tau} \wedge F) = \alpha \in \Gamma$. We claim that $S_{\tau} \le F$. On the contrary, let us assume that $S_{\tau} \wedge (I - F) \ne 0$. Since $\mathcal{T}$ is a separating family of states on $\mathscr{R}$, we may choose a normal tracial state $\rho$ in $\sT$ such that $S_{\tau} \wedge (I - F) \wedge S_{\rho} \ne 0$. Note that for a projection $E$ in $\mathscr{R}$, $\tau(S_{\tau} \wedge E) = 0$ if and only if $S_{\tau} \wedge E = 0$. Thus we have 
\begin{align*}
    \phantom{\Longrightarrow} &\tau \big(S_{\tau} \wedge (F \vee S_{\rho} - F)\big) =\tau(S_{\tau} \wedge (I - F) \wedge S_{\rho}) > 0\\
    \Longrightarrow &\tau(S_{\tau} \wedge (F \vee S_{\rho})) > \tau(S_{\tau} \wedge F) = \alpha.
\end{align*} Clearly $F \vee S_{\rho} \in \sF$, contradicting the fact that $\alpha = \sup \Gamma$. Thus $S_{\tau} \le F$. \hfill $\Diamond$
\end{claimpff}

From Claim \ref{claim:family_sep_states1}, there is a countable collection of states $\{ \tau _n \}_{n \in \N}$ in $\sT$ such that $S_{\tau} \le \bigvee_{n \in \N} S_{\tau _n}$. Define, $$\tau ' := \sum_{n=1}^{\infty} \frac{\tau _n}{2^{n+1}}.$$
Note that $\tau '$ is a normal tracial state on $\mathscr{R}$, and $S_{\tau} \le S_{\tau'}$. From Remark \ref{rmrk:trace_contain}, there is a $\delta ' > 0$ such that for all $\varepsilon > 0$, $$\sO(\tau ', \varepsilon, \delta ') \subseteq \sO(\tau, \varepsilon, \delta).$$ Let $k$ be a positive integer such that $2^{-(k+1)} < \frac{\delta '}{2}.$  

\begin{claim}
\label{clm:finer}
$\bigcap _{i=1}^k \sO(\tau _i, \frac{\varepsilon}{\sqrt{k}}, \frac{2^i}{k} \delta') \subseteq \sO(\tau, \varepsilon, \delta).$
\end{claim}
\begin{claimpff}
Let $A \in \bigcap _{i=1}^k \sO(\tau _i, \frac{\varepsilon}{\sqrt{k}}, \frac{2^i}{k} \delta')$. There are projections $E_1, E_2, \cdots, E_k$ in $\mathscr{R}$ such that $\|A E_i \|\le \frac{\varepsilon}{\sqrt{k}}$ and $\tau _i(I-E_i) \le \frac{2^i}{k} \delta'$ for $1 \le i \le k$. Since $E := \bigvee_{i=1}^k E_i \le \sum_{i=1}^k E_i$, we have 
\begin{align*}
\|AE\|^2 = \big\| A E A^* \big\| &\le \|A(E_1 + \cdots + E_k) A^* \| \\
&\le \sum_{i=1}^k \|A E_i A^*\| = \sum_{i=1}^k \|AE_i \|^2 \\
&\le \varepsilon ^2,
\end{align*}
and 
\begin{align*}
    \tau '(I-E) = \sum_{i=1}^{\infty} 2^{-i-1} \tau_i(I-E) &\le \sum_{i=1}^{k} 2^{-i-1} \tau_i(I-E_i) + \sum_{i=k+1}^{\infty} 2^{-i-1}  \\
    &\le \frac{\delta '}{2} + \frac{\delta '}{2} = \delta '.
\end{align*}
Hence, $A \in \sO(\tau ', \varepsilon, \delta ')$, and we conclude that $$\bigcap _{i=1}^k \sO(\tau _i, \frac{\varepsilon}{\sqrt{k}}, \frac{2^i}{k} \delta') \subseteq \sO(\tau ', \varepsilon, \delta ') \subseteq \sO(\tau, \varepsilon, \delta).$$
\hfill $\Diamond$
\end{claimpff}

From Claim \ref{clm:finer}, we observe that the $\sT$-measure topology is finer than the $\mathfrak{m}$-topology. Clearly the $\sT$-measure topology is coarser than the $\mathfrak{m}$-topology. Thus the $\sT$-measure topology on $\mathscr{R}$ is identical to the $\mathfrak{m}$-topology on $\mathscr{R}$.
\end{proof}

\begin{lem}
\label{lem:main}
\textsl{
Let $\tau$ be a normal tracial state on $\mathscr{R}$ and let $\varepsilon, \delta, \varepsilon_1, \delta_1, \varepsilon_2, \delta_2 > 0$. Then we have the following:
\begin{itemize}
    \item[(i)] $\sO(\tau, \varepsilon, \delta)^* (:= \{A^* : A \in \sO(\tau, \varepsilon, \delta)\}) \subseteq \sO(\tau, \varepsilon, 2\delta)$;
    \item[(ii)] $\sO(\tau, \varepsilon_1, \delta_1) + \sO(\tau, \varepsilon_2, \delta_2) \subseteq \sO(\tau, \varepsilon_1 + \varepsilon_2, \delta_1 + \delta_2)$;
    \item[(iii)] $\sO(\tau, \varepsilon_1, \delta_1) \cdot \sO(\tau, \varepsilon_2, \delta_2) \subseteq \sO(\tau, \varepsilon_1 \varepsilon_2, \delta_1 + \delta_2)$;
    \item[(iv)] $\sV(\tau, \varepsilon_1, \delta_1) + \sV(\tau, \varepsilon_2, \delta_2) \subseteq \sV(\tau, \varepsilon_1 + \varepsilon_2, \delta_1 + \delta_2)$;
    \item[(v)] $\sO(\tau, \varepsilon_1, \delta_1) \cdot \sV(\tau, \varepsilon_2, \delta_2) \subseteq \sV(\tau, \varepsilon_1 \varepsilon_2, 2 \delta_1 + \delta_2)$;
\end{itemize}
}
\end{lem}
\begin{proof}
Before getting into the proofs of the assertions, we prepare the groundwork. For $A \in \mathscr{R}$ and a projection $E$ in $\mathscr{R}$, we define a projection $E_A := I - \sR \big( A(I-E) \big)$ in $\mathscr{R}$. Note that $E_A A(I-E) = 0$ and thus $E_A A = E_A A E$. Using Proposition \ref{prop:range_proj}, (iv), we have 
\begin{align*}
    \phantom{\Longrightarrow } &I-E_A = \mathcal{R}\big(A(I-E)\big) \sim \mathcal{R}\big((I-E)A^*\big) \le I-E\\
    \Longrightarrow &I-E_A \lesssim I-E \\
    \Longrightarrow &\tau(I-E_A) \le \tau(I-E).
\end{align*}
We note the two properties of $E_A$ which we use repeatedly below,
\begin{equation}
\label{eqn:mod_proj1}
    E_A A = E_A A E,
\end{equation}
\begin{equation}
\label{eqn:mod_proj2}
        \tau(I - E_A) \le \tau(I-E).
\end{equation}

\begin{itemize}
    \item[(i)] Let $A \in \sO(\tau, \varepsilon, \delta)$ and $E$ be a projection in $\mathscr{R}$ such that $\|A E \| \le \varepsilon$ and $\tau(I-E) \le \delta.$ Let $F := I - \sR \big( A(I-E) \big) (=: E_A)$. From (\ref{eqn:mod_proj1}), note that $A^*F = EA^*F.$ Since $E(E \wedge F) = F(E \wedge F) = E \wedge F$, we have $A^*(E \wedge F) = EA^*E(E \wedge F).$
Thus $$\|A^*(E \wedge F)\| = \|EA^*E(E \wedge F)\| \le \|EA^*E\| = \|EAE\| \le \|AE\| \le \varepsilon, $$ and using inequality (\ref{eqn:mod_proj2}), we have $$\tau(I - E \wedge F) = \tau\big((I-E) \vee (I-F)\big) \le \tau(I-E) + \tau(I-F) \le 2\delta.$$ Hence $A^* \in \sO(\tau, \varepsilon, 2\delta).$ 
\vskip 0.1in
\item[(ii)] Let $A_1 \in \sO(\tau, \varepsilon_1, \delta_1), A_2 \in \sO(\tau, \varepsilon_2, \delta_2)$. Let $E_1, E_2$ be projections in $\mathscr{R}$ such that $\|A_1 E_1 \| \le \varepsilon_1, \|A_2 E_2\| \le \varepsilon_2$, and $\tau(I - E_1) \le \delta_1, \tau(I-E_2) \le \delta_2$. Then we have $$(A_1 + A_2)(E_1 \wedge E_2) = (A_1 E_1 + A_2 E_2) (E_1 \wedge E_2),$$
which implies that
\begin{align*}
    \|(A_1 + A_2)(E_1 \wedge E_2) \| &\le \varepsilon_1 + \varepsilon_2, \textrm{ and}\\
    \tau(I-E_1 \wedge E_2) = \tau \big((I-E_1) \vee (I-E_2) \big) &\le \tau(I-E_1) + \tau(I-E_2) \le \delta_1 + \delta_2.
\end{align*}
Thus $A_1 + A_2 \in \sO(\tau, \varepsilon_1 + \varepsilon_2, \delta_1 + \delta_2).$
\vskip 0.1in
\item[(iii)] We consider projections $E_1, E_2$ as chosen in part (ii). Let $F :=  I - \sR \big(A_2 ^*(I-E_1)\big)$. From equation (\ref{eqn:mod_proj1}), we observe that $FA_2 ^* = F A_2 ^* E_1$, or equivalently, $A_2 F = E_1 A_2 F$. Thus we have
\begin{align*}
A_1A_2(F \wedge E_2) = A_1 (A_2F) (F \wedge E_2)  &=A_1(E_1 A_2 F)(F \wedge E_2) \\
&= A_1E_1 A_2E_2(F \wedge E_2),
\end{align*}
which implies that
 \begin{align*}
     \|A_1 A_2 (F \wedge E_2) \| &\le \|A_1 E_1 \| \cdot \|A_2 E_2 \| \le \varepsilon_1 \varepsilon_2,\\
     \tau(I - F \wedge E_2) \le \tau(I-F) + \tau(I-E_2) &\le \tau(I-E_1) + \tau(I-E_2) \le \delta_1 + \delta_2,
 \end{align*}
since $\tau(I-F) \le \tau(I-E_1)$  from inequality (\ref{eqn:mod_proj2}).  We conclude that $A_1A_2 \in \sO(\tau, \varepsilon_1 \varepsilon_2, \delta_1 + \delta_2)$.

\item[(iv)] The proof is analogous to that of part (ii). Let $x_1 \in \sV(\tau, \varepsilon_1, \delta_1), x_2 \in \sV(\tau, \varepsilon_2, \delta_2)$. Let $E_1, E_2$ be projections in $\mathscr{R}$ such that $\| E_1x_1 \| \le \varepsilon_1$, $\|E_2 x_2\| \le \varepsilon_2$ and $\tau(I - E_1) \le \delta_1, \tau(I-E_2) \le \delta_2$. Then we have
$$ (E_1 \wedge E_2)(x_1 + x_2) = (E_1 \wedge E_2)(E_1 x_1 + E_2 x_2),$$
which implies that 
\begin{align*}
    \| (E_1 \wedge E_2) (x_1 + x_2)\| &\le \varepsilon_1 + \varepsilon_2, \textrm{ and}\\
    \tau(I - E_1 \wedge E_2) &\le \delta_1 + \delta_2.
\end{align*}
Thus $x_1 + x_2 \in \sV(\tau, \varepsilon_1 + \varepsilon_2, \delta_1 + \delta_2).$

\item[(v)] Let $A \in \sO(\tau, \varepsilon_1, \delta_1), x \in \sV(\tau, \varepsilon_2, \delta_2)$. Since $A^* \in \sO(\tau, \varepsilon_1, 2\delta_1)$ (by part (i)), there is a projection $E$ in $\mathscr{R}$ such that $\|A^* E\| = \|EA \| \le \varepsilon_1$ with $\tau(I-E) \le 2\delta_1$. Let $F$ be a projection in $\mathscr{R}$ such that $\|Fx\| \le \varepsilon_2$ with $\tau(I-F) \le \delta_2$. Let $G := I- \sR (A(I-F))$. From (\ref{eqn:mod_proj1}), noting that $GA = GAF$, we have
$$(E \wedge G)Ax = (E \wedge G)GAx = (E \wedge G)GAFx = (E \wedge G)EAFx.$$
Thus $\| (E \wedge G) Ax\| \le \|EA\| \|Fx\| \le \varepsilon_1 \varepsilon_2$, and since $\tau(I - G) \le \tau(I-F)$ from inequality (\ref{eqn:mod_proj2}), we have  $\tau(I-E \wedge G) \le \tau(I-E) + \tau(I-G) \le 2\delta_1 + \delta_2$.
\end{itemize}
\end{proof}

\begin{cor}
\label{cor:unitary_commutant}
\textsl{
Let $\varepsilon, \delta > 0$ and $\tau$ be a normal tracial state on $\mathscr{R}$. For a unitary operator $U$ in the commutant of $\mathscr{R}$, $U \cdot \sV(\tau, \varepsilon, \delta) = \sV(\tau, \varepsilon, \delta)$.
}
\end{cor}
\begin{proof}
Note that for a projection $E$ in $\mathscr{R}$ and a vector $x \in \mathscr{H}$, if $\|Ex\|\le \varepsilon$, then $\|E(Ux)\| = \|U(Ex)\| = \|Ex\|$ (since $UE = EU$). Keeping this in mind, the assertion follows directly from the definition of $\sV(\tau, \varepsilon, \delta)$.
\end{proof}

\begin{lem}
\label{lem:bounded_measure}
\textsl{Let the nets $\{ A_i \}$ in $\mathscr{R}$, $\{ x_i \}$ in $\mathscr{H}$ be Cauchy in measure. For every normal tracial state $\tau$ on $\mathscr{R}$ and $\delta > 0$, there is an $\alpha > 0$ and an index $j$ such that  $A_i \in \sO(\tau, \alpha, \delta), x_i \in \sV(\tau, \alpha, \delta)$ for all $i \ge j$. }
\end{lem}
\begin{proof}
For $\varepsilon > 0$, choose $j$ in the index set such that for all $i \ge k$, $A_i - A_j \in \sO(\tau, \varepsilon, \frac{\delta}{2})$. Clearly $A_j \in \sO(\tau, \|A_j\| + \varepsilon, \frac{\delta}{2})$ and hence $A_i \in A_j + \sO(\tau, \varepsilon, \frac{\delta}{2}) \subseteq \sO(\tau, 2\varepsilon + \|A_j\|, \delta)$ for all $i \ge j$. We choose $\alpha = 2\varepsilon + \|A_j\|$. The proof for $\{ x_i \}$ is similar.
\end{proof}

\begin{prop}
\label{prop:measure_cauchy_net}
\textsl{
Let the nets $\{ A_i \}$, $\{ B_i \}$ in $\mathscr{R}$ and the nets $\{ x_i \}, \{ y_i \}$ in $\mathscr{H}$ (all with the same index set) be Cauchy in measure. Then we have the following:
\begin{itemize}
    \item[(i)] $\{ A_i ^* \}$ is Cauchy in measure;
    \item[(ii)] $\{ A_i + B_i \}$ is Cauchy in measure;
    \item[(iii)] $\{ A_i B_i \}$ is Cauchy in measure;
    \item[(iv)] $\{ x_i + y_i \}$ is Cauchy in measure;
    \item[(v)] $\{ A_i x_i \}$ is Cauchy in measure.
\end{itemize}
}
\end{prop}
\begin{proof}
Let $\varepsilon, \delta > 0$ and $\tau$ be a normal tracial state on $\mathscr{R}$.

\begin{itemize}
\item[(i)] Let $k$ be an index such that $A_i - A_k \in \sO(\tau, \varepsilon, \frac{\delta}{2})$ for all $i \ge k$. By Lemma \ref{lem:main}, (i), we have $A_i ^* - A_k ^* \in \sO(\tau, \varepsilon, \delta)$ for all $i \ge k$.
\vskip 0.1in

\item[(ii)] Let $k$ be an index such that $A_i - A_k \in \sO(\tau, \frac{\varepsilon}{2}, \frac{\delta}{2}), B_i - B_k \in \sO(\tau, \frac{\varepsilon}{2}, \frac{\delta}{2})$ for all $i \ge k$. By Lemma \ref{lem:main}, (ii), we have $(A_i + B_i) - (A_k + B_k) \in \sO\big(\tau, \varepsilon, \delta \big)$ for all $i \ge k$.
\vskip 0.1in

\item[(iii)]  By Lemma \ref{lem:bounded_measure} for the Cauchy nets $\{ A_i \}, \{ B_i \}$, there is an index $j$ and $\alpha > 0$ such that $A_i \in \sO(\tau, \alpha, \frac{\delta}{6}), B_i \in \sO(\tau, \alpha, \frac{\delta}{6})$ for all $i \ge j$. Let $\varepsilon ' := \sqrt{\alpha ^2 + \varepsilon} - \alpha$ so that $\varepsilon = {\varepsilon '}^2 + 2 \alpha \varepsilon '$. Let the index $j '$ be such that $A_i - A_{j'} \in \sO (\tau,  \varepsilon ', \frac{\delta}{6}), B_i - B_j \in \sO(\tau, {\varepsilon '}, \frac{\delta}{6})$ for all $i \ge j'$. We choose $k := \max \{ j, j'\}$. For $i \ge k$, using Lemma \ref{lem:main}, (ii)-(iii), we have
\begin{align*}
 A_i B_i - A_k B_k = &(A_i - A_k)(B_i - B_k) + A_k(B_i - B_k) + (A_i - A_k)B_k\\
\in &\; \sO\big(\tau, {\varepsilon '}, \frac{\delta}{6}\big) \cdot \sO\big(\tau, \varepsilon ', \frac{\delta}{6}\big) + \sO \big(\tau, \alpha, \frac{\delta}{6}\big) \cdot \sO\big(\tau, \varepsilon ', \frac{\delta}{6}\big) \\
 &+ \sO\big(\tau, \varepsilon ', \frac{\delta}{6}\big) \cdot \sO\big(\tau, \alpha, \frac{\delta}{6}\big)\\
\subseteq & \; \sO(\tau, {\varepsilon '}^2 +  2\alpha \varepsilon ', \delta) \\
=& \;\sO(\tau, \varepsilon, \delta).
\end{align*}
Thus, $A_iB_i - A_k B_k \in \sO(\tau, \varepsilon, \delta)$ for all $i \ge k$.
\vskip 0.1in

\item[(iv)] Let $k$ be an index such that $x_i - x_k \in \sV(\tau, \frac{\varepsilon}{2}, \frac{\delta}{2}), y_i - y_k \in \sV(\tau, \frac{\varepsilon}{2}, \frac{\delta}{2})$ for all $i \ge k$. By Lemma \ref{lem:main}, (iv), we have $(x_i + y_i) - (x_k + y_k) \in \sV\big(\tau, \varepsilon, \delta \big)$ for all $i \ge k$.
\vskip 0.1in

\item[(v)] By Lemma \ref{lem:bounded_measure} for the Cauchy nets $\{ A_i \}$ in $\mathscr{R}$, $\{ x_i \}$ in $\mathscr{H}$, there is an index $j$ and an $\alpha > 0$ such that $A_i \in \sO(\tau, \alpha, \frac{\delta}{9}), x_i \in \sV(\tau, \alpha, \frac{\delta}{9})$ for all $i \ge j$. Let $\varepsilon ' := \sqrt{\alpha ^2 + \varepsilon} - \alpha$. Let the index $j'$ be such that $A_i - A_{j'} \in \sO(\tau, {\varepsilon '}, \frac{\delta}{9}), x_i -x_{j'} \in \sV(\tau, {\varepsilon '}, \frac{\delta}{9})$ for all $i \ge j'$. We choose $k := \max \{ j, j' \}$. For $i \ge k$, using Lemma \ref{lem:main}, (iv)-(v), we have
\begin{align*} 
A_ix_i - A_kx_k = &(A_i - A_k)(x_i - x_k) + A_k (x_i - x_k) + (A_i - A_k)x_k\\
\in &\; \sO \big(\tau, {\varepsilon '}, \frac{\delta}{9} \big) \cdot \sV\big(\tau, {\varepsilon '}, \frac{\delta}{9}\big) + \sO\big(\tau, \alpha, \frac{\delta}{9}\big) \cdot \sV\big(\tau, {\varepsilon '}, \frac{\delta}{9} \big) \\
&+ \sO(\tau, {\varepsilon '}, \frac{\delta}{9}) \cdot \sV(\tau, \alpha,  \frac{\delta}{9})\\
\subseteq &\; \sV(\tau, {\varepsilon '}^2 +  2\alpha \varepsilon ', \delta) \\
= &\; \sV(\tau, \varepsilon, \delta).
\end{align*}
Thus, $A_i x_i - A_k x_k \in \sV(\tau, \varepsilon, \delta)$ for all $i \ge k$.
\end{itemize}
\end{proof}

We denote the completion of $\mathscr{R}$ in the $\mathfrak{m}$-topology (or $\mathfrak{m}$-completion) by $\Wtilde{R}$, and the $\mathfrak{m}$-completion of $\mathscr{H}$ by $\Wtilde{H}$.

\begin{thm}
\label{thm:cont_mvn}
\textsl{
The mappings
\begin{align}
A &\mapsto A^* \textrm{ of } \mathscr{R} \to \mathscr{R},\\
(A, B) &\mapsto A+B \textrm{ of } \mathscr{R} \times \mathscr{R} \to \mathscr{R},\\
(A, B) &\mapsto AB \textrm{ of } \mathscr{R} \times \mathscr{R} \to \mathscr{R},\\
(x, y) &\mapsto x+y \textrm{ of } \mathscr{H} \times \mathscr{H} \to \mathscr{H},\\
(A, x) &\mapsto Ax \textrm{ of } \mathscr{R} \times \mathscr{H} \to \mathscr{H},
\end{align}
are Cauchy-continuous in the $\mathfrak{m}$-topology and thus have unique continuous extensions as mappings of $\Wtilde{R} \, \to \, \Wtilde{R}, \; \Wtilde{R} \! \times \!\! \Wtilde{R} \, \to \, \Wtilde{R}, \; \Wtilde{R} \! \times \!\!\Wtilde{R} \, \to \, \Wtilde{R}, \; \Wtilde{H} \! \times \!\! \Wtilde{H} \, \to \, \Wtilde{H},$ and $\Wtilde{R} \! \times \!\! \Wtilde{H} \, \to \, \Wtilde{H}$, respectively, where $\Wtilde{R}$ is the completion of $\mathscr{R}$ in the $\mathfrak{m}$-topology of $\mathscr{R}$ and $\Wtilde{H}$ is the completion of $\mathscr{H}$ in the $\mathfrak{m}$-topology of $\mathscr{H}$.
}
\end{thm}
\begin{proof}
Follows from Proposition \ref{prop:measure_cauchy_net}.
\end{proof}

With the mappings described in Theorem \ref{thm:cont_mvn}, $\Wtilde{R}$ is a topological $*$-algebra with a continuous representation on the topological vector space $\Wtilde{H}$.

\begin{prop}
\label{prop:hausdorff_mvn}
\textsl{
The Hilbert space $\mathscr{H}$ and the von Neumann algebra $\mathscr{R}$ are Hausdorff spaces in the $\mathfrak{m}$-topology, and thus the canonical mappings into their respective $\mathfrak{m}$-completions $\Wtilde{H}, \, \Wtilde{R}$ are injective.
}
\end{prop}
\begin{proof}
Suppose that $x \in \mathscr{H}$ is in every neighborhood of $0$ in the $\mathfrak{m}$-topology of $\mathscr{H}$. Let $\tau$ be a normal tracial state on $\mathscr{R}$. For each positive integer $n$ there is a projection $F_n$ in $\mathscr{R}$ such that $\|F_nx\| \le 2^{-n}$ and $\tau(I - F_n) \le 2^{-n}$. Define $E_n := \bigwedge_{k=n}^{\infty} F_k.$
Then $\{ E_n \}$ is an increasing sequence of projections in $\mathscr{R}$, $E_n x = 0$ and $\tau(I-E_n) \le 2^{-n+1}$. Let $E$ denote the least upper bound of $\{ E_n \}$. Clearly $Ex = 0$ and by the normality of $\tau$, we have $\tau(I-E) = 0$. Thus $S_{\tau} \le E$ and $x$ is in the nullspace of the support projection of $\tau$. Since this holds for any normal tracial state (and there is a separating family of such states), we conclude that $x = 0$. Therefore $\mathscr{H}$ is Hausdorff in the $\mathfrak{m}$-topology. 

Suppose that $A \in \mathscr{R}$ is in every neighborhood of $0$ in the $\mathfrak{m}$-topology of $\mathscr{R}$. Let $x \in \mathscr{H}$. From Lemma \ref{lem:main}, (v), we observe that $Ax$ is in every neighborhood of $0$ in the $\mathfrak{m}$-topology of $\mathscr{H}$. The discussion in the preceding paragraph leads us to the conclusion that $Ax = 0$. Thus $A = 0$ as $Ax = 0$ for all $x \in \mathscr{H}$. Therefore $\mathscr{R}$ is Hausdorff in the $\mathfrak{m}$-topology.
\end{proof}

\begin{prop}
\label{prop:approx_proj}
\textsl{
Let $A$ be an element of $\Wtilde{R}$. Then for every $\varepsilon > 0$ and faithful normal tracial state $\tau$, there is a projection $E$ in $\mathscr{R}$ such that $AE \in \mathscr{R}$ and $\tau(I-E) \le \varepsilon$.
}
\end{prop}
\begin{proof}
Let $\tau$ be a fixed normal tracial state on $\mathscr{R}$ and $S_{\tau}$ denote the support projection of $\tau$. (Recall that $S_{\tau}$ is a central projection of $\mathscr{R}$.) Let $\{ A_i \}$ be a net in $\mathscr{R}$ which converges in measure to $A$. Clearly, the net $\{ A_i S_{\tau} \}$ in $\mathscr{R}$ converges in measure to $AS_{\tau}$. We wish to extract a subsequence of this net which converges in measure to $A S_{\tau}$. For each positive integer $k$, there is an index $i_k (> i_{k-1})$ such that $A_i - A_{i_k} \in \sO(\tau, 2^{-k}, 2^{-k})$  for all $i \ge i_k$.
\vskip 0.2in

\begin{claim}
\label{clm:net2seq}
\textsl{
The sequence $\{ A_{i_k}S_{\tau} \}_{k \in \N}$ in $\mathscr{R}$ converges in measure to $AS_{\tau}$.
}
\end{claim}
\begin{claimpff}
Let $\varepsilon, \delta > 0$ and $\rho$ be a normal tracial state on $\mathscr{R}$. From Lemma \ref{lem:lebesgue-radon-nikodym}, there are normal tracial states $\rho_1, \rho_2$ on $\mathscr{R}$ and real numbers $t_1, t_2 \ge 0$ with $t_1 + t_2 = 1$ such that $\rho = t_1 \rho_1 + t_2 \rho_2, S_{\rho_1} S_{\tau} = 0,$ and $S_{\rho_2} \le S_{\tau}$. By Lemma \ref{lem:absolute_continuity}, there is a $\delta ' > 0$ such that whenever $\tau(F) \le \delta '$ for a projection $F$ in $\mathscr{R}$, we have $\rho_2(F) \le \delta$.

Let $m$ be a positive integer such that $2^{-m} \le \min \{\varepsilon , \delta ' \}$. For $n \ge m$, let $F_n$ be a projection in $\mathscr{R}$ such that $\|(A_{i_n} - A_{i_m})F_n \| \le 2^{-m} \le \varepsilon$ and $\tau(I-F_n) \le 2^{-m} \le \delta'$. Note that $\rho_2(I-F_n) \le \delta$. Let $G_n := F_n S_{\tau} + I - S_{\tau}$ so that $F_n S_{\tau} = G_n S_{\tau}$. Since $\tau$ is a tracial normal state, we observe that $S_{\tau}$ is a central projection of $\mathscr{R}$. We have $\big( (A_{i_n}-A_{i_m})F_n \big) S_{\tau} = \big( (A_{i_n}-A_{i_m})S_{\tau} \big) F_n = (A_{i_n}S_{\tau} - A_{i_m} S_{\tau} ) G_n$. Thus
$$\|(A_{i_n}S_{\tau} - A_{i_m} S_{\tau} ) G_n \| = \|\big( (A_{i_n}-A_{i_m})F_n \big) S_{\tau}\| \le 2^{-m} \le \varepsilon.$$
Furthermore, since $\tau(S_{\tau}(I-F_n)) \le \tau(I-F_n) \le \delta'$, we have $\rho_2 (S_{\tau}(I-F_n)) \le \delta$. Thus 
\begin{align*}
    \rho(I-G_n) &= t_1 \rho_1(I-G_n) + t_2 \rho_2(I-G_n) \\
    &= t_1 \rho_1 \big(S_{\tau}(I-F_n)\big) + t_2 \rho_2 \big(S_{\tau}(I-F_n)\big) \le \delta.
\end{align*}
In summary, $(A_{i_n}S_{\tau} - A_{i_m} S_{\tau} ) \in \sO(\rho, \varepsilon, \delta) $ for $n \ge m$. Thus the sequence $\{ A_{i_k} S_{\tau} \}$ in $\mathscr{R}$ is Cauchy in measure and converges to $A S_{\tau}$ (as the net $\{ A_i S_{\tau} \}$ in $\mathscr{R}$ converges to $AS_{\tau}$.) \hfill $\Diamond$
\end{claimpff}
\vskip 0.05in
We reuse notation from the proof of Claim \ref{clm:net2seq}. Let $\varepsilon > 0$. For $k \ge 1$, we have $$\|(A_{i_{k+1}}S_{\tau} - A_{i_k}S_{\tau})F_k \| \le 2^{-k}, $$ and $\tau(I-F_k) \le 2^{-k}$. Let $$E_n := \bigwedge_{k=n}^{\infty} F_k.$$
The sequence of projections $\{ E_n \}$ in $\mathscr{R}$ is increasing and $\tau(I-E_n) \le 2^{-n+1}$. Fix $m$ such that $2^{-m+1} \le \varepsilon$. For $k \ge m$, since $E_m = F_k E_m$, we have $$\|(A_{i_{k+1}}S_{\tau} - A_{i_k} S_{\tau})E_m \| \le \|(A_{i_{k+1}}S_{\tau} - A_{i_k} S_{\tau})F_k \|  \le 2^{-k}.$$
This shows that the sequence $\{ A_{i_k}S_{\tau}E_m \}_{k \in \N}$ in $\mathscr{R}$ is Cauchy in norm and thus converges in norm to an operator $B$ in $\mathscr{R}$. Since the norm topology is finer than the $\mathfrak{m}$-topology, we conclude that the sequence converges to $B$ in measure. Thus $B = AS_{\tau}E_m$. Choosing $E := S_{\tau}E_m$, we see that $AE \in \mathscr{R}$ and $\tau(I-E) = \tau(I-S_{\tau}E_m) = \tau(I-E_m) \le 2^{-m+1} \le \varepsilon$.  
\end{proof}

\begin{cor}
\label{cor:approx_proj}
\textsl{
Let $A$ be an element of $\Wtilde{R}$. Then there is an increasing sequence of projections $\{ E_n \}$ in $\mathscr{R}$ converging in measure to $I$ with $AE_n \in \mathscr{R}$ for every $n \in \N$.}
\end{cor}
\begin{proof}
Let $\sT$ be a collection of normal tracial states on $\mathscr{R}$ with mutually orthogonal support projections. From the proof of Proposition \ref{prop:approx_proj}, for each $\tau \in \sT$, there is an increasing sequence of projections $\{ E _ {\tau, n} \}$ converging in measure to $S_{\tau}$ such that the sequence $\{ A E_{\tau, n} \}$ is in $\mathscr{R}$. For $n \in \N$, let $$E_n := \bigvee_{\tau \in \sT} E_{\tau, n} = \sum_{\tau \in \sT} E_{\tau, n}.$$ 
Note that $E_n \uparrow \sum_{\tau \in \sT} S_{\tau} = I$ in measure and $AE_n \in \mathscr{R}$ for every $n \in \N$. 
\end{proof}
\section{Functorial Approach to Murray-von Neumann algebras}

In this section, we show that Murray-von Neumann algebras arise as $\mathfrak{m}$-completions of finite von Neumann algebras. With this intrinsic description at hand, we explore the appropriate notion of morphism between Murray-von Neumann algebras, the appropriate notion of Murray-von Neumann {\it subalgebra}, and the order structure. This enables us to view Murray-von Neumann algebras intrinsically as ordered complex topological $*$-algebras.

\begin{lem}
\label{lem:equality_proj}
\textsl{
Let $E_1$ and $E_2$ be projections in $\mathscr{R}$ such that for every $\varepsilon > 0$ and normal tracial state $\tau$ there is a projection $F$ in $\mathscr{R}$ with $\tau(I-F) \le \varepsilon$ and $E_1 \wedge F = E_2 \wedge F$. Then $E_1 = E_2$.
}
\end{lem}
\begin{proof}
Let $E, F$ be projections in $\mathscr{R}$ such that $E \wedge F = 0$. From the Kaplansky formula (see \cite[Theorem 6.1.7]{kadison-ringrose2}), we note that $E = I - (I - E) = (I-E) \vee (I-F) - (I-E) \sim 
(I-F) - (I-E) \wedge (I-F).$ Thus $E \precsim I - F$ and for any tracial state $\tau$ on $\mathscr{R}$, we have $\tau(E) \le \tau(I-F)$.

Let $E_1, E_2, F$ be projections in $\mathscr{R}$ as described in the statement of the lemma (with $F$ depending on $\varepsilon$ and $\tau$). Since $E_1 \wedge F = E_2 \wedge F$, we have $E_1 \wedge F = (E_1 \wedge E_2) \wedge F$ which implies that $(E_1 - E_1 \wedge E_2) \wedge F = 0$. From the discussion in the preceding paragraph, we observe that $\tau(E_1 - E_1 \wedge E_2) \le \tau(I-F) \le \varepsilon$. Thus $\tau(E _1 - E_1 \wedge E_2) = 0$ for every normal tracial state $\tau$ on $\mathscr{R}$, and since the family of such states is separating, we conclude that $E_1 - E_1 \wedge E_2 = 0$. By a symmetric argument, we have $E_2 - E_1 \wedge E_2 = 0$. Hence $E_1 = E_2$.
\end{proof}

\begin{lem}
\label{lem:sep_proj}
\textsl{
Let $A$ and $B$ be closed densely-defined operators affiliated with $\mathscr{R}$. Suppose that for every $\varepsilon > 0$ and faithful normal tracial state $\tau$ there is a projection $E$ in $\mathscr{R}$ with $\tau(I-E) \le \varepsilon$ such that $E\mathscr{H} \subseteq \mathscr{D}(A) \cap \mathscr{D}(B)$ and $Ax = Bx$ for all $x \in E\mathscr{H}$. Then $A=B$.
}
\end{lem}
\begin{proof}
Let $M_2(\mathscr{R})$ denote the von Neumann algebra of all operators acting on the Hilbert space $\mathscr{H}_2 \;(:= \mathscr{H} \oplus \mathscr{H})$ given by $2 \times 2$ matrices with entries in $\mathscr{R}$. As $A$ and $B$ are closed operators, their graphs are closed subspaces of $\mathscr{H}_2$. Let $\textbf{G}_A$ and $\textbf{G}_B$ be the projections onto the graphs of $A$ and $B$, respectively. Recall that every unitary operator in $M_2(\mathscr{R})'$ is of the form $\begin{bmatrix}
U & 0\\
0 & U
\end{bmatrix}$ for some unitary operator $U$ in $\mathscr{R}'$. For a vector $x \in \mathscr{D}(A)$ and $U \in \mathscr{R}'$, we have $$\begin{bmatrix}
U & 0\\
0 & U
\end{bmatrix}
\begin{bmatrix}
x\\
Ax
\end{bmatrix} =
\begin{bmatrix}
Ux\\
UAx
\end{bmatrix} =
\begin{bmatrix}
Ux\\
AUx
\end{bmatrix}.$$
Thus the graph of $A$ is invariant under the action of any unitary operator in $M_2(\mathscr{R})'$, and by the double commutant theorem, $\textbf{G}_A \in M_2(\mathscr{R})$. Similarly $\textbf{G}_B \in M_2(\mathscr{R})$.

For a normal tracial state $\tau$ on $\mathscr{R}$, we define a corresponding normal tracial state $\text{\boldmath$\tau _2$}$ on $M_2(\mathscr{R})$ by
$$
\text{\boldmath$\tau _2$} \Big(
\begin{bmatrix}
A_{11} & A_{12}\\
A_{21} & A_{22}
\end{bmatrix} \Big) = 
\frac{\tau(A_{11}) + \tau(A_{22})}{2}.
$$
(We remind the reader that every normal tracial state on $M_2(\mathscr{R})$ arises in this manner.) Let $\textbf{E}_2 := \begin{bmatrix}
E & 0 \\
0 & E
\end{bmatrix}$, and $\textbf{I}_2 := \begin{bmatrix}
I & 0\\
0 & I
\end{bmatrix}$ be the identity operator in $M_2(\mathscr{R})$. By the hypothesis, we have $\textbf{G}_A \wedge \textbf{E} = \textbf{G}_B \wedge \textbf{E}$ and $\text{\boldmath$\tau _2$}(\textbf{I}_2 - \textbf{E}_2) \le \varepsilon$. Thus by Lemma \ref{lem:equality_proj}, we have $\textbf{G}_A = \textbf{G}_B$ which implies that $A = B$.
\end{proof}

If $A$ is an element of $\Wtilde{R}$ and $x \in \mathscr{H}$, then from Theorem \ref{thm:cont_mvn} we know that $Ax \in \; \Wtilde{H}$. If $Ax$ is in $\mathscr{H}$, then we say that $x$ is in the domain of the operator of multiplication by $A$. We symbolically denote this by $M_A x = Ax$ and say that $x \in \mathscr{D}(M_A)$.

\begin{thm}
\label{thm:aff_op_complete}
\textsl{
\begin{itemize}
    \item[(i)] For every $A \in \; \Wtilde{R}$, $M_A$ is a closed densely-defined operator affiliated with $\mathscr{R}$;
    \item[(ii)] For every closed densely-defined operator $T$ affiliated with $\mathscr{R}$, there is an element $A \in \; \Wtilde{R}$ such that $T = M_A$;
    \item[(ii)] For $A \in \; \Wtilde{R}$, $M_{A}^* = M_{A^*}$;
    \item[(iii)] For $A, B \in \; \Wtilde{R}$, $M_{A + B} = \overline{M_A + M_B}$;
    \item[(iv)] For $A, B  \in \; \Wtilde{R}$, $M_{AB}  = \overline{M_A M_B}$.
\end{itemize}
Thus the mapping $A \mapsto M_A : \,\Wtilde{R} \to \afr$ is a $*$-isomorphism between $\Wtilde{R}$ and $\afr$ extending the identity mapping from $\mathscr{R}$ to $\mathscr{R}$.
}
\end{thm}
\begin{proof}
(i) Let $A \in \; \Wtilde{R}$ and $U$ be a unitary operator in $\mathscr{R}'$, the commutant of $\mathscr{R}$. Consider a net $\{ A_i \}$ in $\mathscr{R}$ which converges in measure to $A$. For every vector $x \in \mathscr{H}$, by Theorem \ref{thm:cont_mvn}, the nets $\{ A_i x \}, \{ A_i Ux \}$ in $\Wtilde{H}$ respectively converge in measure to $Ax, AUx$. From Corollary \ref{cor:unitary_commutant}, the net $\{ UA_i x \}$ in $\Wtilde{H}$ converges in measure to $UAx$. (We emphasize that the preceding assertion does not follow from Theorem \ref{thm:cont_mvn} as $U$ may not belong to $\mathscr{R}$.) Let $x \in \mathscr{D}(M_A)$ so that $Ax \in \mathscr{H}$. Since $UA_ix = A_i Ux$ for every index $i$, taking limits in the $\mathfrak{m}$-topology we conclude that $AUx = UAx \in \mathscr{H}$. Thus if $x \in \mathscr{D}(M_A)$, then $Ux \in \mathscr{D}(M_{A} )$. In other words, for a unitary operator $U$ in $\mathscr{R}'$, $U \cdot \mathscr{D}(M_A) = \mathscr{D}(M_A)$ and $U M_A = M_A U$. We conclude that $M_A$ is affiliated with $\mathscr{R}$.

Let $\{ x_i \}$ be a net of vectors in $\mathscr{D}(M_A)$ such that $x_i \rightarrow 0$ in norm, and $M_A x_i \rightarrow y$ in norm. Then $x_i \rightarrow 0$ in measure  so that, by Theorem \ref{thm:cont_mvn}, $M_A x_i \rightarrow 0$ in measure. Since $M_A x_i \rightarrow y$ in norm and thus in measure, and $\Wtilde{H}$ is Hausdorff by Proposition \ref{prop:hausdorff_mvn}, we conclude that $y=0$. Thus $M_A$ is closed.

By Proposition \ref{prop:approx_proj}, for $\varepsilon > 0$ and normal tracial state $\tau$ on $\mathscr{R}$ there is a projection $E$ in $\mathscr{R}$ such that $\tau(I-E) \le \varepsilon$ and $E\mathscr{H} \subseteq \mathscr{D}(M_A)$ or equivalently, $\mathscr{D}(M_A)^{\perp} \subseteq (I-E)\mathscr{H}$. Hence $\mathscr{D}(M_A)^{\perp} = 0$ and $M_A$ is densely-defined.
\vskip 0.2in

\noindent (ii)  Let $E_n \in \mathscr{R}$ denote the spectral projection of $|T| := (T^*T)^{\frac{1}{2}}$ corresponding to the 
interval $[0, n]$. Note that $\bigcup_{n=1}^{\infty} E_n \mathscr{H}$ is a core for $T$ and $E_n \uparrow I$ in measure. For $m \ge n$, clearly $(TE_m -  TE_n) E_n = 0$. Thus the sequence $\{ TE_n \}$ in $\mathscr{R}$ is Cauchy in measure and converges to an element $A$ of $\Wtilde{R}$. Moreover, for $x \in \bigcup_{n=1}^{\infty} E_n \mathscr{H}$, the sequence of vectors $\{ TE_nx \}$ eventually takes the constant value $Tx \in \mathscr{H}$ and also converges in measure to $Ax$. Hence $\bigcup_{n=1}^{\infty} E_n \mathscr{H} \subseteq M_A$ and $Tx = Ax$ for $x \in \bigcup_{n=1}^{\infty} E_n \mathscr{H}$. Thus $T = M_A$.
\vskip 0.2in

\noindent (iii) For $\varepsilon > 0$ and a faithful normal tracial state $\tau$, by Proposition \ref{prop:approx_proj}, there is a projection $E$ in $\mathscr{R}$ such that $A^*E \in \mathscr{R}$ (hence, $EA \in \mathscr{R}$) with $\tau(I-E) \le \varepsilon$. Thus $E\mathscr{H} \subseteq M_{A^*}$. For a vector $x$ in $\mathscr{D}(M_A)$ and a vector $y$ in $E\mathscr{H}$, we have $$\langle M_A x, y \rangle = \langle M_Ax, Ey \rangle = \langle EAx, y \rangle = \langle x, A^* Ey \rangle = \langle x, M_{A^*}y \rangle,$$
so that $y \in \mathscr{D}(M_{A} ^*)$ and $M_A ^* y = M_{A^*}y$. Thus $E\mathscr{H} \subseteq \mathscr{D}(M_{A^*}) \cap \mathscr{D}(M_A ^*)$ and for all $x \in \mathscr{H}$, we have $M_A ^* x = M_{A^*} x$. From Lemma \ref{lem:sep_proj}, we conclude that $M_A ^* = M_{A^*}.$
\vskip 0.2in

\noindent (iv) Clearly $M_A + M_B \subseteq M_{A+B}$ and since $M_{A+B}$ is closed, $M_A + M_B$ is pre-closed. For $\varepsilon > 0$ and a faithful normal tracial state $\tau$, by Proposition \ref{prop:approx_proj},  there are projections $E, F$ in $\mathscr{R}$ such that $AE, BF \in \mathscr{R}$ with $\tau(I-E) \le \frac{\varepsilon}{2}, \tau(I-F) \le \frac{\varepsilon}{2}$. Note that $A(E\wedge F) = AE(E \wedge F) \in \mathscr{R}, B(E \wedge F) = BF(E \wedge F) \in \mathscr{R}$ and $\tau(I-E\wedge F) \le \varepsilon$. Note that $(E \wedge F)\mathscr{H} \subseteq \mathscr{D}(\overline{M_A + M_B}) \cap \mathscr{D}(M_{A+B})$ and for $x \in (E\wedge F)\mathscr{H}$, we have $(M_A + M_B)x = \overline{M_A + M_B}\,x = M_{A+B}\,x$. By Lemma \ref{lem:sep_proj}, $\overline{M_A + M_B} = M_{A+B}$.
\vskip 0.2in

\noindent (v) We reuse notation from part (iv). Clearly $M_A M_B \subseteq M_{AB}$ and since $M_{AB}$ is closed, $M_A M_B$ is pre-closed. Let $G := \sN \big( (I-E)BF \big) = I - \sR \big( FB^*(I-E) \big)$ so that $x \in G\mathscr{H}$ if and only if $BFx \in E \mathscr{H}$ (note $E \mathscr{H} \subseteq \mathscr{D}(M_A)$). Thus $(F \wedge G) \mathscr{H} \subseteq \mathscr{D}(\overline{M_A M_B}) \cap \mathscr{D}(M_{AB})$. From (\ref{eqn:mod_proj2}), we observe that $\tau(I - F \wedge G) \le \varepsilon$, and for all $x \in (F \wedge G)\mathscr{H}$, we have $M_A M_B x = \overline{M_A M_B} x = M_{AB}x$. By Lemma \ref{lem:sep_proj}, $\overline{M_A M_B} = M_{AB}$.
\end{proof}

\begin{lem}
\label{lem:close_to_zero}
\textsl{
Let $\mathscr{R}$ be a finite von Neumann algebra and $\tau$ be a normal tracial state on $\mathscr{R}$. Let $H$ be a positive operator in $\mathscr{R}$ and $\varepsilon > 0$ such that $\tau(H) \le \varepsilon$. Then there is a projection $E$ in $\mathscr{R}$ such that $\|H E \| \le \sqrt{\varepsilon}$ and $\tau(I-E) \le \sqrt{\varepsilon}$. In other words, $H \in \sO(\tau, \sqrt{\varepsilon}, \sqrt{\varepsilon})$.
}
\end{lem}
\begin{proof}
Let $E_{\lambda}$ denote the spectral projection of $H$ corresponding to the interval $[0, \lambda]$ for $\lambda \ge 0$. From \cite[Theorem 5.2.2, (iv)]{kadison-ringrose1}, for $\mu  \ge 0$, we observe that
\begin{align*}
    HE_{\mu}  \le \mu E_{\mu},\\
    H(I-E_{\mu}) \ge \mu (I-E_{\mu}).    
\end{align*}
Choosing $\mu = \sqrt{\varepsilon}$, we have 
\begin{align*}
    HE_{\sqrt{\varepsilon}} \le \sqrt{\varepsilon} E_{\sqrt{\varepsilon}},\\
    H(I - E_{\sqrt{\varepsilon}}) \ge \sqrt{\varepsilon} (I - E_{\sqrt{\varepsilon}}).
\end{align*}

Thus $\|H E_{\sqrt{\varepsilon}} \| \le \sqrt{\varepsilon}$ and $\sqrt{\varepsilon} \tau (I - E_{\sqrt{\varepsilon}}) \le \tau \big( H(I - E_{\sqrt{\varepsilon}}) \big) \le \tau(H) \le \varepsilon $, which implies that $\tau(I - E_{\sqrt{\varepsilon}}) \le \sqrt{\varepsilon}$.
\end{proof}

\begin{cor}
\label{cor:meas_cont_of_sup_conv}
\textsl{
Let $\{ H_i \}$ be an increasing net of self-adjoint operators in $\mathscr{R}$ converging to $H$ in the strong-operator topology (or equivalently, in the ultra-weak topology). Then $H_i \uparrow H$ in the $\mathfrak{m}$-topology.
}
\end{cor}
\begin{proof}
Note that $H = \sup H_i$. Let $\tau$ be a normal tracial state on $\mathscr{R}$. From Lemma \ref{lem:close_to_zero}, if $|\tau(H) - \tau(H_i)| = |\tau (H - H_i)|\le \varepsilon$, then $(H - H_i) \in \sO(\tau, \sqrt{\varepsilon}, \sqrt{\varepsilon})$. By the normality of $\tau$, we have $$\lim_i \tau(H_i) = \tau(H).$$ Thus $H_i \uparrow H$ in the $\mathfrak{m}$-topology on $\mathscr{R}$.
\end{proof}

\begin{remark}
We make frequent use of Corollary \ref{cor:meas_cont_of_sup_conv} in the context of a fixed self-adjoint operator $A$ in $\afr$ for the increasing sequence of projections $\{ E_n \}$  where $E_n$ denotes the spectral projection of $A$ corresponding to the interval $[-n, n]$ for $n \in \N$. 
\end{remark}

\begin{remark}
It is natural to wonder how the $\mathfrak{m}$-topology on $\mathscr{R}$ relates to other intrinsic topologies such as the norm topology and ultraweak topology on $\mathscr{R}$. In \cite[Theorem 3.17]{nayak-matrix}, it is shown that the $\mathfrak{m}$-topology is coarser than the norm topology. Furthermore, it is shown via counterexamples in the setting of $II_1$ factors that the $\mathfrak{m}$-topology is neither finer nor coarser than the ultraweak topology.
\end{remark}

\begin{cor}
\label{cor:vn_subalg}
\textsl{
Let $\mathscr{S}$ be a $*$-subalgebra of $\mathscr{R}$ which is closed in the $\mathfrak{m}$-topology and contains the same identity as $\mathscr{R}$. Then $\mathscr{S}$ is a von Neumann subalgebra of $\mathscr{R}$.
}
\end{cor}
\begin{proof}
Let $\{ H_i \}$ be a monotonically increasing net of self-adjoint elements of $\mathscr{S}$ that is bounded above. By Corollary \ref{cor:meas_cont_of_sup_conv}, $\{ H_i \}$ converges in measure to $\sup H_i$. Since $\mathscr{S}$ is $\mathfrak{m}$-closed and $H_i \uparrow \sup H_i$ in measure, we observe that $\sup H_i \in \mathscr{S}$. Thus $\mathscr{S}$ is a von Neumann subalgebra of $\mathscr{R}$.
\end{proof}

\begin{thm}
\label{thm:functorial}
\textsl{
Let $\mathscr{R}_1$ and $\mathscr{R}_2$ be finite von Neumann algebras and $\Phi : \mathscr{R}_1 \to \mathscr{R}_2$ be a unital $*$-homomorphism. Then the following are equivalent:
\begin{itemize}
    \item[(i)] $\Phi$ is normal;
    \item[(ii)] $\Phi$ is Cauchy-continuous for the $\mathfrak{m}$-topologies on $\mathscr{R}_1$ and $\mathscr{R}_2$.
\end{itemize}
}
\end{thm}
\begin{proof}
(i) $\Longrightarrow$ (ii). Let $\tau_2$ be a normal tracial state on $\mathscr{R}_2$. Since $\Phi$ is normal, we note that $\tau_1 := \tau_2 \circ \Phi$ is a normal tracial state on $\mathscr{R}_1$. For $\varepsilon, \delta > 0$, let $A \in \sO(\tau_1, \varepsilon, \delta)$ so that there is a projection $E$ in $\mathscr{R}_1$ such that $\|AE \| \le \varepsilon$ and $\tau_1(I-E) \le \delta$. Note that $\|\Phi(A) \Phi(E) \|  = \| \Phi(AE)\| \le \varepsilon$, $\tau_2(I - \Phi(E)) = \tau_1(I-E) \le \delta$ and $\Phi(E)$ is a projection in $\mathscr{R}_2$. Consequently, $$\Phi \big( \sO(\tau_1, \varepsilon, \delta) \big) \subseteq \sO(\tau_2, \varepsilon, \delta).$$
Thus if a net $\{ A_i \}$ in $\mathscr{R}_1$ is Cauchy in measure, then the net $\{ \Phi(A_i) \}$ in $\mathscr{R}_2$ is also Cauchy in measure. We conclude that $\Phi$ is Cauchy-continuous for the $\mathfrak{m}$-topologies on $\mathscr{R}_1$ and $\mathscr{R}_2$.
\vskip 0.1in
(ii) $\Longrightarrow$ (i). Let $\{ H_i \}_{i \in \Lambda}$ be an increasing net of self-adjoint operators in $\mathscr{R}_1$ with supremum $H$. Without loss of generality, we may assume that $\{ H_i \}$ is a net of positive operators in $\mathscr{R}_1$ (considering the net $\{ H_i - H_1 \}$ if necessary). By Corollary \ref{cor:meas_cont_of_sup_conv}, $H_i \uparrow H$ in the $\mathfrak{m}$-topology on $\mathscr{R}_1$. By our hypothesis, $\Phi(H_i) \uparrow \Phi(H)$ in the $\mathfrak{m}$-topology on $\mathscr{R}_2$. 

For the sake of brevity, define $K_i := \Phi(H - H_i)$. Note that $K_i \downarrow 0$ in the $\mathfrak{m}$-topology on $\mathscr{R}_2$, and $0 \le K_i \le \Phi(H)$ for all $i \in \Lambda$. Let $\tau_2$ be a faithful normal tracial state on $\mathscr{R}_2$. Let $F$ be a projection in $\mathscr{R}_2$ such that $\|K_{i}F \| \le \varepsilon$, $\tau_2(I-F) \le \delta$. Using the Cauchy-Schwarz inequality in the context of the sesquilinear form given by $\langle A, B \rangle = \tau_2(B^*A)$ for $A, B \in \mathscr{R}$, and noting that $\|K_i\| \le \|\Phi(H)\| \le \|H\|$, we have
\begin{align*}
    \tau_2(K_i) &= \tau_2(K_iF) + \tau_2(K_i(I-F)) \\
    & \le \|K_i F\| + \sqrt{\tau_2(K_i^2)} \sqrt{\tau_2(I-F)}\\
    & \le \|K_i F\| + \sqrt{\|K_i^2\|} \sqrt{\tau_2(I-F)}\\
    & \le \varepsilon + \|H\| \cdot \sqrt{\delta}.
\end{align*}

In other words, if $K_i \in \sO(\tau_2, \varepsilon, \delta)$, then $\tau_2(K_i) \le \varepsilon + \|H\|\sqrt{\delta}$. Since $K_i \downarrow 0$ in the $\mathfrak{m}$-topology, we conclude that $\lim_i \tau_2(K_i) = 0$. Let $K := \inf K_i$ (the infimum exists as the decreasing net $\{ K_i \}$ is bounded below by $0$). Since $\tau_2$ is normal, we conclude that $\tau_2(K) = 0$. As the family of normal tracial states on $\mathscr{R}_2$ is separating, we conclude that $K = 0$, which implies that $\sup \Phi(H_i) = \Phi(H) = \Phi(\sup H_i)$. Thus $\Phi$ is normal.
\end{proof}

From Theorem \ref{thm:functorial}, a normal unital $*$-homomorphism between finite von Neumann algebras lifts to a $\mathfrak{m}$-continuous unital $*$-homomorphism between the corresponding Murray-von Neumann algebras. Thus $\mathfrak{m}$-completion gives rise to a functor from the category of finite von Neumann algebras to the category of unital complex topological $*$-algebras.

\begin{prop}[Order Structure]
\label{prop:mtilde_pos}
\textsl{
\begin{itemize}
\item[(i)] $\mathscr{R}^{\,\textrm{sa}}$ is $\mathfrak{m}$-closed in $\mathscr{R}$. The $\mathfrak{m}$-closure of $\mathscr{R}^{\,\textrm{sa}}$ in $\Wtilde{R}$ is $\Wtilde{R}^{\!\!\textrm{sa}}$, the set of self-adjoint elements in $\Wtilde{R}$.
\item[(ii)] $\mathscr{R}^{+}$ is $\mathfrak{m}$-closed in $\mathscr{R}$. Let $\Wtilde{R}^{\!\!+}$ denote the $\mathfrak{m}$-closure of $\mathscr{R}^{+}$ in $\Wtilde{R}$. Then $\Wtilde{R}^{\!\!+}$ is a cone in $\Wtilde{R}$.
\end{itemize}
The cone $\,\Wtilde{R}^{\!\!+}$ equips $\, \Wtilde{R} ^{\!\!\mathrm{sa}}$ with a natural order structure making $(\Wtilde{R}; \Wtilde{R}^{\!\!+})$ an ordered complex topological $*$-algebra.
}
\end{prop}
\begin{proof}
See proof of \cite[Proposition 3.8]{nayak-matrix} (and use Corollary \ref{cor:approx_proj}). 
\end{proof}

\begin{prop}
\label{prop:pos_conj}
\textsl{
 $A \in \; \Wtilde{R}^{\!\!+}$ if and only if $M_A$ is a positive operator. 
}
\end{prop}
\begin{proof}
See proof of \cite[Proposition 3.10]{nayak-matrix} (and use Theorem \ref{thm:aff_op_complete}, Corollary \ref{cor:approx_proj}).
\end{proof}

\begin{remark}
\label{rmrk:raff_pos}
From Proposition \ref{prop:pos_conj}, note that the $*$-algebra isomorphism $A \mapsto M_A :\, \Wtilde{R} \to \afr$ also induces an order isomorphism between $(\Wtilde{R};\Wtilde{R}^{\!\!+}) $ and $(\afr; \afr ^{+})$. Thus we may transfer the topology of $\Wtilde{R}$ to $\afr$ and view $\afr$ as a unital ordered complex topological $*$-algebra.  We think of $\Wtilde{R}$ as an abstract Murray-von Neumann algebra, and  $\afr$ as a concrete representation of $\Wtilde{R}$.
\end{remark}

\begin{definition}
The $\mathfrak{m}$-completion of an abstract finite von Neumann algebra is said to be an {\it abstract Murray-von Neumann algebra.}
\end{definition}

\begin{remark}
An operator $A \in \afr$ is bounded (that is, $A \in \mathscr{R}$) if and only if $A^*A \le \lambda ^2 I$ for some $\lambda \in \R_{+}$. Thus the underlying finite von Neumann algebra of an abstract Murray-von Neumann algebra may be extracted using the order structure. 
\end{remark}

\begin{thm}[see {\cite[Proposition 4.3]{nayak-matrix}}]
\label{thm:matrix-iso}
\textsl{
For a positive integer $n$, $M_n(\afr)$ and $M_n(\mathscr{R})_{\textrm{aff}}$ are isomorphic as unital ordered complex topological $*$-algebras with the isomorphism extending the identity mapping of $M_n(\mathscr{R})$ to $M_n(\mathscr{R})$.
}
\end{thm}
\begin{proof}
Let $\tau$ be a normal tracial state on $\mathscr{R}$, and $\text{\boldmath$\tau _n$}$ denote the corresponding normal tracial state on $M_n(\mathscr{R})$ given by $$\text{\boldmath$\tau _n$}({\bf A}) = \frac{\tau(A_{11}) + \cdots + \tau(A_{nn})}{n}, \textrm{ for } {\bf A} = (A_{ij}) \in M_n(\mathscr{R}).$$ For $\varepsilon_{ij}, \delta_{ij} > 0 \; (1 \le i, j \le n)$, from \cite[Proposition 4.3]{nayak-matrix}, we have
\begin{equation}
\prod_{1 \le i, j \le n} \sO(\tau, \varepsilon_{ij}, \delta_{ij}) \subseteq \sO \Big(\text{\boldmath$\tau _n$}, \sum_{1 \le i, j \le n} \varepsilon_{ij}, \sum_{1 \le i, j \le n} \delta_{ij} \Big).
\end{equation}
For $0 < \varepsilon$ and $0 < \delta < \frac{1}{16n} $, again from \cite[Proposition 4.3]{nayak-matrix}, we have 
\begin{equation}
\sO(\text{\boldmath$\tau _n$}, \varepsilon, \delta) \subseteq \prod_{i=1}^{n^2} \sO \big(\tau, 2\varepsilon, \sqrt{4n \delta}\big).
\end{equation}
Thus the product topology on $M_n(\mathscr{R})$ (viewed as $\mathscr{R} \times \overset{n^2}{\cdots} \times \mathscr{R}$) derived from the $\mathfrak{m}$-topology on $\mathscr{R}$ is identical to the $\mathfrak{m}$-topology on $M_n(\mathscr{R})$ considered as a finite von Neumann algebra. The assertion follows by noting that $M_n(\afr)$ may be viewed as the $\mathfrak{m}$-completion of $M_n(\mathscr{R})$.
\end{proof}

\begin{lem}
\label{lem:dom_state}
\textsl{
Let $\mathscr{R}$ be a finite von Neumann algebra and $\omega$ be a normal state on $\mathscr{R}$. Then there is a normal tracial state $\tau$ on $\mathscr{R}$ such that $S_{\omega} \le S_{\tau}$.
}
\end{lem}
\begin{proof}
Let $\{ \tau_i \}_{i \in \Lambda}$ be a separating family of normal tracial states on $\mathscr{R}$ with mutually orthogonal supports. For an index $i \in \Lambda$, we observe that $\omega (S_{\omega} \wedge S_{\tau_i}) = 0$ if and only if $S_{\omega} \wedge S_{\tau_i} = 0$. Since $\omega$ is normal and $\{ S_{\tau _i} \}_{i \in \Lambda}$ are mutually orthogonal projections, we have $$\sum_{i \in \Lambda} \omega(S_{\omega} \wedge S_{\tau _i}) = \omega( S_{\omega} \wedge (\bigvee_{i \in \Lambda} S_{\tau _i})
  ) = \omega(S_{\omega}) = 1.$$ Thus $S_{\omega} \wedge S_{\tau _i} \ne 0$ for at most countably many indices $i$ (which we relabel as $1,2, \cdots $). For the normal tracial state $\tau$ on $\mathscr{R}$
 given by $$\tau := \frac{\tau _1}{2} + \frac{\tau _2}{4} + \cdots + \frac{\tau _n}{2^{n+1}} + \cdots ,$$
we have $S_{\omega} \le S_{\tau }$.
\end{proof}

\begin{prop}
\label{prop:compatibility_subalgebra}
\textsl{
Let $\mathscr{R}$ be a finite von Neumann algebra and $\mathscr{S}$ be a von Neumann subalgebra of $\mathscr{R}$. Then the $\mathfrak{m}$-topology on $\mathscr{S}$ is identical to the subspace topology induced from the $\mathfrak{m}$-topology on $\mathscr{R}$. 
}
\end{prop}
\begin{proof}
 Let $\sT _{\mathscr{R}}$ be a separating family of normal tracial states on $\mathscr{R}$. Then 
 $\sT _{\mathscr{S}} := \{ \restr{\tau}{\mathscr{S}}: \tau \in \sT \}$ is a separating family of normal
 tracial states on $\mathscr{S}$. For $\varepsilon, \delta > 0$ and $\tau \in \sT _{\mathscr{R}}$, it is
 easy to see that $\sO _{\mathscr{S}}(\restr{\tau}{\mathscr{S}}, \varepsilon, \delta) \subseteq \sO_{\mathscr{R}}(\tau, \varepsilon, \delta) \cap \mathscr{S}.$ Thus by Lemma \ref{lem:family_sep_states},
 the $\mathfrak{m}$-topology on $\mathscr{S}$ is finer than the subspace topology induced from the
 $\mathfrak{m}$-topology on $\mathscr{R}$.

The proof of the converse requires more effort as we need to show that the projection in the definition of the fundamental system of neighborhoods of $0$ may be chosen in $\mathscr{S}$.  From \cite[Theorem 2.15, Proposition 2.36]{takesaki1}, there is a faithful normal conditional expectation $\Phi$ from $\mathscr{R}$ onto $\mathscr{S}$. For a normal tracial state $\tau$ on $\mathscr{S}$, let $\Phi ^*(\tau)$ denote the pullback of $\tau$ via $\Phi$, that is, $\Phi ^* (\tau) \big( A \big) = \tau \big( \Phi(A) \big)$ for all $A \in \mathscr{R}$. Note that since $\Phi$ is a unital normal map, $\Phi ^*(\tau)$ is a normal state on $\mathscr{R}$.

\begin{claim}
For $\varepsilon, \delta > 0$ and a normal tracial state $\tau$ on $\mathscr{S}$, there is $\delta ' > 0$ and a normal tracial state $\rho$ on $\mathscr{R}$ such that $\sO _{\mathscr{R}} \big( \rho, \frac{\varepsilon}{\sqrt{2}}, \delta ' \big) \cap \mathscr{S} \subseteq \sO _{\mathscr{S}}(\tau , \varepsilon, \delta).$
\end{claim}
\begin{claimpff}
From Lemma \ref{lem:dom_state}, we may choose a normal tracial state $\rho$ on $\mathscr{R}$ such that $S_{\Phi ^* (\tau)} \le S_{\rho}$. By Lemma \ref{lem:absolute_continuity}, there is $\delta ' > 0$ such that whenever $\rho(F) \le \delta '$ for a projection $F$ in $\mathscr{R}$, we have $\Phi^*(\tau) \big (F \big) \le \frac{\delta}{4}$.

Let $A \in \sO _{\mathscr{R}} ( \rho, \frac{\varepsilon}{\sqrt{2}}, \delta ') \cap \mathscr{S}$ so that there is a projection $E$ in $\mathscr{R}$ with $\| AE \| \le \frac{\varepsilon}{\sqrt{2}}$ and $\rho (I-E) \le \delta '$. Thus $\| A \Phi(E)\| = \| \Phi(AE)\| \le \|AE\|\le \frac{\varepsilon}{\sqrt{2}}$ and $\tau(I - \Phi(E)) = \Phi ^*(\tau) \big(I - E \big) \le \frac{\delta}{4}.$ Let $H := \Phi(E)^2$. Note that $0 \le H \le I$ and hence the spectrum of $H$ lies in $[0, 1]$. 

For $0 \le \lambda \le 1$, let $E_{\lambda}$ denote the spectral projection of $H$ corresponding to the interval $[0, \lambda]$. Since $H$ is in $\mathscr{S}$, the spectral projections $E_{\lambda}$ are also in $\mathscr{S}$ for all $\lambda \in [0, 1]$. As in the proof of Lemma \ref{lem:close_to_zero}, for $\mu \in [0,1]$ we note that $\mu (I - E_\mu) \le H (I - E_{\mu}) \le I - E_{\mu}$ and $HE_{\mu} \le \mu E_{\mu}$. Moreover, for any self-adjoint operator $T$ in $\mathscr{R}$, we have $\tau\big((T-I)^2\big) \ge 0$ from which it follows that $\tau(I - T^2) \le \tau(2I - 2T)$. Putting these observations together, we have
\begin{align*}
\tau(E_{\mu}) \le \tau \big(I - H(I - E_{\mu}) \big) &\le \tau(I-H) + \tau(H E_{\mu})\\
&=\tau(I-\Phi(E)^2) + \tau(H E_{\mu})\\
& \le \tau \big( 2I - 2\Phi(E) \big) + \mu \tau(E_{\mu}) \\
& \le \frac{\delta}{2} + \mu \tau(E_{\mu}),
\end{align*} which implies (for $\mu \ne 1$) that $$\tau(E_{\mu}) \le \frac{\delta}{2(1-\mu)}.$$ Furthermore, since $\mu(I - E_{\mu}) \le H(I-E_{\mu}) \le H$, we observe that
\begin{align*}
    \mu \| A(I-E_{\mu}) \|^2 &= \| A \big(\mu (I-E_{\mu})\big) A^* \| \le  \|A H(I-E_{\mu}) A^* \| \\
    &\le \|AHA^*\| =  \|A \Phi(E)\|^2 \le \frac{\varepsilon ^2}{2}.
\end{align*}
Choosing $\mu = \frac{1}{2}$ and setting $F = I- E_{\frac{1}{2}}$, we have $\| AF \| \le \varepsilon, \tau(I-F) \le \delta,$ and $F \in \mathscr{S}$. Thus $A \in \sO _{\mathscr{S}}(\tau, \varepsilon, \delta)$. \hfill $\Diamond$
\end{claimpff}

From the above claim, we conclude that the subspace topology on $\mathscr{S}$ induced from the $\mathfrak{m}$-topology on $\mathscr{R}$ is finer than the $\mathfrak{m}$-topology on $\mathscr{S}$. In summary, the two topologies are identical.
\end{proof}

For a complex ordered $*$-algebra $\fA$, we define $$\fA _{\textrm{nb}} := \{ A \in \fA : \exists \lambda \in \R_{+} \textrm{ such that } A^*A \le \lambda ^2 I \}.$$
Let $\fM$ be a Murray-von Neumann algebra. Note that the underlying finite von Neumann algebra of $\fM$ is given by $\fM _{\textrm{nb}}$. Let $\fN$ be a $\mathfrak{m}$-closed $*$-subalgebra of $\fM$ containing the same identity as $\fM$. We may consider $\fN$ as a complex ordered $*$-algebra with the positive cone $\fN ^{+} := \fN \cap \fM ^{+}$. Clearly $\fN _{\textrm{nb}} = \fN \cap \fM _{\textrm{nb}}$. Let $\{ A_i \}$ be a net in $\fN _{\textrm{nb}}$ which converges in measure to $A \in \fM _{\textrm{nb}}$. As $\fN$ is $\mathfrak{m}$-closed, $A \in \fN$. Thus $A$ is in $\fN \cap \fM _{\textrm{nb}} = \fN _{\textrm{nb}}$. This shows that $\fN _{\textrm{nb}}$ is a $\mathfrak{m}$-closed $*$-subalgebra of $\fM _{\textrm{nb}}$. From Corollary \ref{cor:vn_subalg}, $\fN _{\textrm{nb}}$ is a von Neumann subalgebra of $\fM _{\textrm{nb}}$. Furthermore, by Proposition \ref{prop:compatibility_subalgebra}, we have that $\fN$ is the $\mathfrak{m}$-completion of $\fN _{\textrm{nb}}$. This leads us to the appropriate notion of subobject in the category of Murray-von Neumann algebras.

\begin{definition}
An $\mathfrak{m}$-closed $*$-subalgebra of a Murray-von Neumann algebra $\fM$ containing the same identity as $\fM$ is said to be a {\it Murray-von Neumann subalgebra} of $\fM$. 
\end{definition}

\begin{remark}
\label{rmrk:mvn_subalgebra}
From the preceding discussion, we observe that Murray-von Neumann subalgebras of a Murray-von Neumann algebra $\fM$ correspond to von Neumann subalgebras of the underlying finite von Neumann algebra of $\fM$.
\end{remark}

\begin{cor}
\textsl{
Let $\mathscr{R}$ be a finite von Neumann algebra with center $\mathscr{C}$. The center of $\afr$ is the Murray-von Neumann algebra $\mathscr{C}_{\textrm{aff}}$.
}
\end{cor}
\begin{proof}
It is straightforward to see that $\mathscr{C}_{\textrm{aff}}$ is contained in the center of $\afr$. Since the center of $\afr$ is a Murray-von Neumann subalgebra of $\afr$, its underlying finite von Neumann algebra is contained in $\mathscr{C}$ which implies that the center of $\afr$ is contained in $\mathscr{C}_{\textrm{aff}}$. 
\end{proof}

\begin{prop}[Monotone completeness]
\label{prop:lub_mvna}
\textsl{
Let $\mathscr{R}$ be a finite von Neumann algebra. If $\{ H_i \}_{n \in \N}$ is an increasing net of self-adjoint operators in $\afr$ bounded above by a self-adjoint operator $H$ in $\afr$, then $\{ H_i \}$ has a least upper bound in $\afr$. Furthermore, $\{ H_i \}$ converges in measure to $\sup H_i$.
}
\end{prop}
\begin{proof}
For $n \in \N$, let $E_n$ be the spectral projection of $H$ corresponding to the interval $[-n, n]$. For a fixed $k \in \N$, note that $\{ E_k \afrprod H_i \afrprod E_k \}_{i \in \Lambda}$ is a monotonically increasing net of self-adjoint elements in $\mathscr{R}$ bounded above by $E_k \afrprod  H \afrprod E_k \in \mathscr{R}$. Let $\varepsilon, \delta >0$ and $\tau$ be a faithful tracial state on $\mathscr{R}$. Since $E_n \uparrow I$ in the $\mathfrak{m}$-topology, there is an $m \in \N$ such that $$A - E_m \afrprod A \afrprod E_m = A \afrprod (I \afrdiff E_m) \afrsum (I \afrdiff E_m ) \afrprod A \afrprod E_m \in \sO(\tau, \frac{\varepsilon}{3}, \frac{\delta}{3})$$ for all $A \in \afr$. By Corollary \ref{cor:meas_cont_of_sup_conv}, the net $\{ E_m \afrprod H_i \afrprod E_m \}_{i \in \Lambda}$in $\mathscr{R}$ is Cauchy in measure and hence there is an index $j$ such that $$E_m \afrprod H_i \afrprod E_m - E_m \afrprod H_j \afrprod E_m \in \sO(\tau, \frac{\varepsilon}{3}, \frac{\delta}{3}),$$ for all $i \ge j$. Thus 
\begin{align*}
&\phantom{=} H_i - H_j \\
&= (H_i \afrdiff E_m \afrprod  H_i \afrprod E_m ) \afrsum (E_m \afrprod H_i \afrprod E_m \afrdiff E_m \afrprod H_j \afrprod E_m) \afrsum (E_m \afrprod H_j \afrprod E_m \afrdiff H_j) \\
&\in \sO(\tau, \varepsilon, \delta),
\end{align*}
for all $i \ge j$. Hence $\{ H_i \}$ is Cauchy in measure and converges to an operator $\bar{H}$ in $\afr$. Let $G$ be a self-adjoint operator in $\afr$ such that $H_i \le G$ for all indices $i$. Recall that the set of positive operators in $\afr$ is closed in the $\mathfrak{m}$-topology. Thus $\bar{H} = \lim_i H_i \le G $ which implies that $\bar{H}$ is the least upper bound of $\{ H_i \}$.
\end{proof}

\begin{remark}
Proposition \ref{prop:lub_mvna} tells us that the self-adjoint part of a Murray-von Neumann algebra is {\it monotone complete}.
\end{remark}

\begin{definition}
Let $\fM _1, \fM _2$ be Murray-von Neumann algebras. An adjoint-preserving homomorphism $\Phi : \fM _1 \to \fM _2$ is said to be {\it normal} if for every bounded monotonically increasing net $\{ H_i \}$ of self-adjoint elements in $\fM _1$ with an upper bound, we have $\Phi(\sup H_i) = \sup \Phi(H_i)$ in $\fM _2$.
\end{definition}

\begin{thm}
\label{thm:normal_mvna}
\textsl{
Let $\fM _1$ and $\fM _2$ be Murray-von Neumann algebras and $\Phi : \fM _1 \to \fM _2$ be a unital $*$-homomorphism. Then the following are equivalent:
\begin{itemize}
    \item[(i)] $\Phi$ is normal;
    \item[(ii)] $\Phi$ is $\mathfrak{m}$-continuous.
\end{itemize}
}
\end{thm}
\begin{proof}
Let $\mathscr{R} _1, \mathscr{R} _2,$ respectively, be finite von Neumann algebras whose corresponding Murray-von Neumann algebras are $\fM_1, \fM_2$, respectively. As noted in \S \ref{sec:intro}, $\Phi$ restricts to a unital $*$-homomorphism between $\mathscr{R}_1$ and $\mathscr{R}_2$.

(i) $\Rightarrow$ (ii). If $\Phi$ is normal, then $\restr{\Phi}{\mathscr{R}_1} : \mathscr{R}_1 \to \mathscr{R}_2$ is also normal. By Theorem \ref{thm:functorial}, $\restr{\Phi}{\mathscr{R}_1}$ is continuous in the $\mathfrak{m}$-topologies on $\mathscr{R}_1$ and $\mathscr{R}_2$, and thus $\Phi$ is $\mathfrak{m}$-continuous.

(ii) $\Rightarrow$ (i). Let $\{ H_i \}$ be a monotonically increasing net of self-adjoint elements in $\fM _1$ with supremum $H$. Since the net $\{ \Phi(H_i) \}$ in $\fM _2$ has an upper bound $\Phi(H)$, it has a least upper bound $K \in \fM _2$. From Proposition \ref{prop:lub_mvna}, we observe that $H_i \uparrow H$ in the $\mathfrak{m}$-topology of $\fM _1$ and $\Phi(H_i) \uparrow K$ in the $\mathfrak{m}$-topology of $\fM _2$. From the $\mathfrak{m}$-continuity of $\Phi$, we conclude that $$K = \lim _i \Phi(H_i) = \Phi(\lim _i H_i) = \Phi(H).$$
Thus $\Phi$ is normal.
\end{proof}

From the above theorem, we may consider the morphisms in the category of Murray-von Neumann algebras to be normal unital $*$-homomorphisms. Thus every morphism between Murray-von Neumann algebras arises from a morphism between their underlying finite von Neumann algebras.

\section{Abelian Murray-von Neumann Algebras}
\label{sec:abelian_mvn}

Let $X$ be a locally compact Hausdorff space with a positive Radon measure $\mu$. Let $L^0 (X; \mu)$ denote the algebra of $\mu$-measurable functions on $X$. For a compact set $K \subseteq X$, $\varepsilon, \delta >0$, let
$$ \sO(K, \varepsilon, \delta) := \{ f \in L^0(X; \mu) : \mu(\{ x \in K : |f(x)| > \varepsilon \}) \le \delta \}$$
The translation-invariant topology generated by the fundamental system of neighborhoods of $0$ given by $\{ \sO(K, \varepsilon, \delta) \}$ is called the {\it topology of local convergence in measure} or simply, {\it measure topology}.

\begin{prop}
\label{prop:complete_L0}
\textsl{
With the measure topology, $L^{0}(X; \mu)$ is a complete topological $*$-algebra.
}
\end{prop}
\begin{proof}
Let the net $\{ f_i \}$ in $L^0(X; \mu)$ be Cauchy in measure, and $F \subseteq X$ be the set of points $x$ in $X$ for which $\{ f_i(x) \}$ is a Cauchy net in $\C$. We define a function $f$ on $X$ as follows,
$$
f(x) :=
\left\{
	\begin{array}{ll}
		\lim_i f_i(x) & \mbox{if } x \in F\\
		0 & \mbox{if } x \in X \backslash F \textrm{      }.
	\end{array}
\right.$$

Let $K$ be a compact subset of $X$. We may choose a subsequence $\{ f_n \}$ such that $\mu(\{ x \in K : |f_{n+1}(x) - f_n(x)| > 2^{-n} \}) \le 2^{-n}$. For $n \in \N$, let $$S_n := \{ x \in K : |f_{n+1}(x) - f_n(x)| \le 2^{-n} \}.$$
Note that $S_n$ is $\mu$-measurable. Consider the $\mu$-measurable subset of $K$ defined by $$S := \bigcup_{n = 1}^{\infty} \big( \bigcap_{i = n}^{\infty} S_i \big).$$
Clearly $K \backslash S$ is $\mu$-null and for $x \in S$, $\{ f_n(x) \}$ is a Cauchy sequence in $\C$. 
Since $f \mathbbm{1} _{S}$ is the pointwise limit of the sequence of $\mu$-measurable functions given by  $\{ f_n \mathbbm{1} _{S} \}$ and $K \backslash S$ is $\mu$-null, we note that $f \mathbbm{1} _K$ is $\mu$-measurable (see \cite[Corollaries to Theorem 1.14]{rudin2}).
Hence $f$ is locally $\mu$-measurable, and since $\mu$-measurability is a local property (see discussion following \cite[Definition 7, Chapter IX-\S 1.5]{bourbaki-int2}), we conclude that $f$ is $\mu$-measurable. It is easy to verify that $\{ f_i \}$ converges in measure to $f$.
\end{proof}

Let $L^{\infty}(X; \mu)$ denote the abelian von Neumann algebra of all essentially bounded $\mu$-measurable functions on $X$. It is usually represented on the Hilbert space $L^2(X; \mu)$ via multiplier operators. For a non-$\mu$-null compact subset $K$ of $X$, let $\mu_K : L^{\infty}(X;\mu) \to \C$ be the normal state on $L^{\infty}(X;\mu)$ defined by the mapping $f \mapsto \frac{1}{\mu(K)} \int_{K} f \; d \mu$ (normality of $\mu _K$ follows from the fact that  $\mathbbm{1} _{K} \in L^1(X;\mu)$). Note that $\{ \mu _K \}$ is a separating family of normal tracial states on $L^{\infty}(X ;\mu)$. Since $\mu_K(\mathbbm{1} _F) = \frac{\mu(K \cap F)}{\mu(K)}$ for every $\mu$-measurable set $F \subseteq X$, from Lemma \ref{lem:family_sep_states} we see that the subspace topology induced on $L^{\infty}(X; \mu)$ from the measure topology on $L^0(X; \mu)$ is the $\mathfrak{m}$-topology on $L^{\infty}(X; \mu)$.  

 Let $f \in L^0(X; \mu)$, and $S_n := \{ x \in X : f(x) \le n \}$ for $n \in \N$. We observe that $\{ S_n \}$ is an increasing sequence of $\mu$-measurable subsets of $X$ with $\mathbbm{1} _{S_n} \uparrow \mathbbm{1}_X$ in the weak-$*$ topology (and hence by Corollary \ref{cor:meas_cont_of_sup_conv}, in the $\mathfrak{m}$-topology.) The net $\{ f \mathbbm{1} _{S_n} \}$ in $L^{\infty}(X; \mu) \subseteq L^0(X; \mu)$ is Cauchy in the $\mathfrak{m}$-topology and converges in measure to $f$. In other words, $L^{\infty}(X; \mu)$ is dense in $L^0(X; \mu)$. It is not hard to see from the definitions that the measure topology on $L^0(X;\mu)$ is the one obtained from the completion of the $\mathfrak{m}$-topology on $L^{\infty}(X; \mu)$.
 The positive cone of $L^0(X; \mu)$ is given by the set of all essentially positive $\mu$-measurable functions on $X$, which is the closure of the set of essentially positive $\mu$-measurable functions in $L^{\infty}(X;\mu)$ in the measure topology. Hence $$\big( L^{\infty}(X; \mu) \big)_{\textrm{aff}} \cong L^0(X; \mu),$$
 as unital ordered complex topological $*$-algebras.

\begin{remark}
\label{rmrk:abelian_mvn}
Every abelian von Neumann algebra is of the form $L^{\infty}(X; \mu)$ for a locally compact Hausdorff space $X$ with a positive Radon measure $\mu$ (see Theorem 1.18 in \cite[Chapter III]{takesaki1}). Hence, from the preceding discussion, we note that every abelian Murray-von Neumann algebra is of the form $L^0(X; \mu)$.
\end{remark}

\section{Borel Function Calculus and Operator Inequalities}
\label{sec:approx_meas}

In this section, we algebraically define the notions of {\it spectrum}, {\it point spectrum} for elements in a Murray-von Neumann algebra $\fM$. This gives us a definition that is independent of the spatial representation of $\fM$. We define the Borel function calculus for normal elements in $\fM$ and note that it is compatible with the usual definition of the Borel function calculus for (unbounded) normal operators acting on a Hilbert space (see \cite[\S 5.6]{kadison-ringrose1}). For a normal element $A$ in $\fM$, the two-variable complex polynomial functions in $A, A^*$ are shown to be $\mathfrak{m}$-dense in the space of Borel functions in $A$ allowing for the use of polynomial approximation techniques.

\begin{definition}
\label{def:spectrum}
Let $A$ be an element of a Murray-von Neumann algebra $\fM$. We say that a complex number $\lambda$ is a {\it spectral value} for $A$ (relative to $\fM$) when $A \afrdiff \lambda I$ does not have a two-sided inverse in $\fM _{\textrm{nb}}$, the underlying finite von Neumann algebra of $\fM$. The set of spectral values of $A$ is called the {\it spectrum} of $A$ and is denoted by $\mathrm{sp} ^{\fM}(A)$. 

We say that a complex number $\lambda$ is an {\it eigenvalue} of $A$ (relative to $\fM$) when $A \afrdiff \lambda I$ does not have a two-sided inverse in $\fM$. The set of eigenvalues of $A$ is called the {\it point spectrum} of $A$ and is denoted by $\mathrm{sp} _e ^{\fM}(A)$. (If $\fM := \afr$ where $\mathscr{R}$ is a finite von Neumann algebra acting on the Hilbert space $\mathscr{H}$, recall from Proposition \ref{prop:luck} that $\lambda \in \mathrm{sp}_e ^{\fM}(A)$ if and only if $A - \lambda I$ has non-trivial nullspace. This explains our usage of the term `eigenvalue'.)
\end{definition}

\begin{prop}
\label{prop:spec_inv}
\textsl{
Let $\fM$ be a Murray-von Neumann algebra and $\fN$ be a Murray-von Neumann subalgebra of $\fM$. For $A \in \fN$, we have 
\begin{itemize}
    \item[(i)] $\mathrm{sp} ^{\fM}(A) = \mathrm{sp} ^{\fN}(A)$;
    \item[(ii)] $\mathrm{sp}_e ^{\fM}(A) = \mathrm{sp}_e ^{\fN}(A)$.
\end{itemize}
}
\end{prop}
\begin{proof}
Let $\mathscr{R}$ be a finite von Neumann algebra acting on the Hilbert space $\mathscr{H}$ such that $\fM \cong \afr$. By Remark \ref{rmrk:mvn_subalgebra}, there is a von Neumann subalgebra $\mathscr{S}$ of $\mathscr{R}$ such that $\fN \cong \mathscr{S}_{\textrm{aff}}$.
\vskip 0.1in

\noindent (i) If $\lambda \in \mathrm{sp} ^{\fM}(A)$, then $A$ does not have an inverse in $\mathscr{R}$; this implies that it has no inverse in $\mathscr{S}$. Hence $\mathrm{sp} ^{\fM}(A) \subseteq \mathrm{sp} ^{\fN}(A)$. In order to establish the inclusion $\mathrm{sp} ^{\fN}(A) \subseteq \mathrm{sp} ^{\fM}(A)$, we show below that: if $A \in \fN$, and $A$ has an inverse $B$ in $\mathscr{R}$, then $B \in \mathscr{S}$.

First we assume that $A$ is self-adjoint. Let $U$ be a unitary operator in $\mathscr{S}'$ so that $UA = AU$. For a vector $x \in \mathscr{D}(A)$, we have $BU(Ax) =  BAUx = Ux =  UB(Ax)$. From Proposition \ref{prop:luck}, since $A$ has dense range in $\mathscr{H}$,  we observe that $BU = UB$. Thus by the double commutant theorem, $B \in \mathscr{S}$. 

We next consider an element $A \in \fN$ (not necessarily self-adjoint) with an inverse $B \in 
\mathscr{R}$. Then $A^* \in \fN$ with inverse $B^*$ in $\mathscr{R}$. Since $A^* \afrprod A$ has inverse
$B \afrprod B^*$ in $\mathscr{S}$, we conclude from the preceding paragraph that $BB^* \in \mathscr{S}$. Thus $B = (B \afrprod B^*) \afrprod A^* \in \fN _{\textrm{nb}} = \mathscr{S}$.
\vskip 0.1in

\noindent (ii) The proof is similar to that of part (i) and is left to the reader.
\end{proof}

In view of Proposition \ref{prop:spec_inv}, we omit the superscripts in the definition of the spectrum and the point spectrum of an element in a Murray-von Neumann algebra.

\begin{prop}
\textsl{
For an element $A$ of a Murray-von Neumann algebra $\fM$, $\mathrm{sp} (A)$ is a closed subset of $\C$. 
}
\end{prop}
\begin{proof}
Let $\lambda \notin \mathrm{sp} (A)$ so that $B := (A \afrdiff \lambda I)^{-1} \in \fM _{\textrm{nb}}$. For $\alpha \in \C$, we have $A \afrdiff \alpha I = \big( I \afrdiff (\alpha -\lambda)(A \afrdiff \lambda I)^{-1} \big) \afrprod (A \afrdiff \lambda I) = (I \afrdiff (\alpha -\lambda)B) \afrprod (A \afrdiff \lambda I)$. If $|\alpha - \lambda | < \|B\|$, then $(I \afrdiff (\alpha -\lambda)B)$ is invertible in $\fM$ with bounded inverse and hence $(A \afrdiff \alpha I)$ has an inverse in $\fM _{\mathrm{nb}}$. Thus $\alpha \notin \mathrm{sp}(A)$. We conclude that $\C \backslash \mathrm{sp}(A)$ is open or equivalently, $\mathrm{sp}(A)$ is closed.
\end{proof}

\begin{prop}
\textsl{
Let $\mathscr{R}$ be a finite von Neumann algebra and $A$ be a normal element in $\afr$. Then there is a smallest von Neumann subalgebra $\mathscr{A}$ of $\mathscr{R}$ such that $A \in \mathscr{A}_{\textrm{aff}}$. Furthermore, $\mathscr{A}$ is abelian.
}
\end{prop}
\begin{proof}
Let $\fA$ be the Murray-von Neumann subalgebra of $\afr$ given by the $\mathfrak{m}$-closure of the $*$-subalgebra of $\fM$ consisting of polynomials in $A, A^*$. Clearly $\fA$ is the smallest Murray-von Neumann subalgebra of $\afr$ which contains $A$. Furthermore, $\fA$ is abelian as $A$ is normal. By Remark \ref{rmrk:mvn_subalgebra}, there is a von Neumann subalgebra $\mathscr{A}$  of $\mathscr{R}$ such that $\fA  \cong \mathscr{A}_{\textrm{aff}}$, and $\mathscr{A}$ is the smallest von Neumann subalgebra of $\mathscr{R}$ such that $A \in \mathscr{A}_{\textrm{aff}}$. Since $\fA$ is abelian, $\mathscr{A}$ being a subalgebra is also abelian.
\end{proof}

For a subset $S$ of $\C$, we denote the set of {\it bounded} Borel functions on $S$ by $\mathscr{B}_b(S)$, and the set of Borel functions on $S$ by $\mathscr{B}_u(S)$.

\begin{lem}
\textsl{
Let $X$ be a locally compact Hausdorff space with a positive Radon measure $\mu$. For $g \in L^0(X; \mu)$, we have $\mu \big( g^{-1}(\C \backslash \mathrm{sp}(g) ) \big) = 0$. Hence the spectrum of a normal element in a Murray-von Neumann algebra is non-empty.
}
\end{lem}
\begin{proof}
Let $\lambda \notin \mathrm{sp}(g)$ so that $\frac{1}{g - \lambda \mathbbm{1} _X} \in L^{\infty}(X; \mu)$. There is an $\varepsilon > 0$ such that $|g - \lambda \mathbbm{1} _X| \ge \varepsilon \mathbbm{1}_X$. Thus for the open set $O := \{ z \in \C : |z - \lambda | < \frac{\varepsilon}{2}\} \subseteq \C$, we have $\mu(g^{-1}(O)) = 0$. We can find a countable collection of such $O$'s that cover the open set $\C \backslash \mathrm{sp}(g)$. Thus $\mu \big( g^{-1}(\C \backslash \mathrm{sp}(g) ) \big) = 0$. Since $\mu(X) \ne 0$, we conclude that $\sigma(g)$ is non-empty.
\end{proof}

\begin{prop}
\label{prop:borel-fun-calc}
\textsl{
Let $X$ be a locally compact Hausdorff space and with a positive Radon measure $\mu$. Let $g \in L^0(X; \mu)$ and without loss of generality, assume that $g$ takes values in $\mathrm{sp}(g)$ (after modifying $g$ on a $\mu$-null set if necessary). There is a unique $\sigma$-normal $*$-homomorphism $\Psi _B : \mathscr{B}_u \big(  \mathrm{sp} (g) \big) \to L^0(X; \mu)$ mapping the constant function $1$ to $\mathbbm{1} _X$ and the identity transformation $\iota$ on $\mathrm{sp}(g)$ ($\iota(z) = z$) onto $g$. More specifically, the mapping $\Psi _B$ is given by $\Psi _B(f) = f \circ g$.
}
\end{prop}
\begin{proof}
Since $\sigma(g)$ is closed, note that for a Borel measurable function $f$ on $\mathrm{sp}(g)$, $f \circ g$ is a $\mu$-measurable function on $X$, that is, $f \circ g  \in L^0(X; \mu)$. First we show that the mapping $\Psi _B(f) = f \circ g$ is in fact a $\sigma$-normal $*$-homomorphism between $ \mathscr{B}_u \big(  \mathrm{sp} (g) \big)$ and $L^0(X; \mu)$. It is easy to verify that $\Psi_B$ is a $*$-homomorphism which sends the constant function $1$ to $\mathbbm{1} _X$ and the identity transformation $\iota$ on $\mathrm{sp}(g)$ to $g$. The key part to prove is that $\Psi _B$ is $\sigma$-normal. Let $\{ f_n \}$ be an increasing sequence of Borel functions on $\mathrm{sp}(g)$ tending pointwise to the Borel function $f$. Then $\{ f_n \circ g \}$ is an increasing sequence of functions in $L^0(X; \mu)$ bounded above by $f \circ g$. By Proposition \ref{prop:lub_mvna}, the increasing sequence $\{ f_n \circ g \}$ has a least upper bound $h$ in $L^0(X; \mu)$. Let $$X_n := \{ x \in X : f_n(g(x)) \le h(x) \}, X_0 := \{ x \in X : h(x) \le f(g(x)) \}.$$
For $ x \in \bigcap _{k=0}^{\infty} X_k$, clearly $f(g(x)) = h(x)$. Since $X \backslash X_k$ is a $\mu$-null set for $k \in \N \cup \{ 0 \}$, we conclude that $h = f \circ g$ $\mu$-almost everywhere and thus $h = f \circ g$ in $L^0(X; \mu)$. Hence $\Psi _B$ is $\sigma$-normal.

We next prove the uniqueness of the mapping. Let $\Phi : \mathscr{B}_u \big(  \mathrm{sp} (g) \big) \to L^0(X; \mu)$ be another mapping with the prescribed properties. Recall that  $ \mathscr{B}_b \big(  \mathrm{sp} (g) \big)$ is a $C^*$-algebra with complex conjugation as involution and the sup-norm. From the $\sigma$-normality of $\Psi _B$ and $\Phi$, we observe that both of these mappings are adjoint-preserving and order-preserving. As $\Psi _B(\mathbbm{1} _{\mathrm{sp}(g)}) = \Phi(\mathbbm{1} _{\mathrm{sp}(g)}) = I$, we note that the mappings $\Psi _B$ and $\Phi$ take $ \mathscr{B}_b \big(  \mathrm{sp} (g) \big)$ into $L^{\infty}(X; \mu)$, and are norm-decreasing. Since $\Phi(\iota) = \Psi_B (\iota)$, we have $$ \Psi _B \big( (1 + |\iota|^2)^{-1} \big) = \Psi _B (1 + |\iota|^2)^{-1} = \Phi (1 + |\iota|^2)^{-1} = \Phi \big( (1 + |\iota|^2)^{-1} \big).$$
Similarly, $$\Psi _B \big( \iota (1 + |\iota|^2)^{-1} \big) = \Phi \big( \iota (1 + |\iota|^2)^{-1} \big).$$
By the Stone-Weierstrass theorem, the $*$-subalgebra of $\mathscr{B}_b \big(  \mathrm{sp} (g) \big)$ generated by $\iota (1 + |\iota|^2)^{-1}, (1 + |\iota|^2)^{-1}$ and the constant function $1$ is norm-dense in $C_0(\mathrm{sp}(g))$. Thus $\Psi_B(f) = \Phi(f)$ for all $f \in C_0(\mathrm{sp}(g))$. 

For an open subset $O \subseteq \mathrm{sp}(g)$, there is an increasing sequence of positive continuous functions in $C_0(\mathrm{sp}(g))$ converging pointwise to $\mathbbm{1} _{O}$. (Use Urysohn's lemma appropriately.) From the $\sigma$-normality of $\Psi _B$ and $\Phi$, we see that $\Psi _B(\mathbbm{1} _{O}) = \Phi(\mathbbm{1} _{O})$. Let $\mathscr{F}$ be the family of Borel sets $S$ of $\mathrm{sp}(A)$ such that $\Psi _B(\mathbbm{1} _{S}) = \Phi(\mathbbm{1} _{S})$. We have shown that all open subsets of $\mathrm{sp}(g)$ are in $\mathscr{F}$. It is easy to see that if $S \in \mathscr{F}$, then $\mathrm{sp}(g) \backslash S \in \mathscr{F}$. As $\Psi _B$ and $\Phi$ are $\sigma$-normal, we note that $\mathscr{F}$ is a $\mathrm{sp}$-algebra. Thus $\mathscr{F}$ is in fact the full family of Borel sets of $\mathrm{sp}(g)$. Since every function $f \in \mathscr{B}_b \big(  \mathrm{sp} (g) \big)$ is a norm-limit of step functions, we conclude that $\Psi _B (f) = \Phi(f)$ for $f \in \mathscr{B}_b \big(  \mathrm{sp} (g) \big)$. Note that $(1 + |f|^2)^{-1} , f(1 + |f|^2)^{-1} \in \mathscr{B}_b(\mathrm{sp}(g))$ for $f \in \mathscr{B}_u(\mathrm{sp}(g))$. Thus for all $f \in \mathscr{B}_u(\mathrm{sp}(g))$, we have
\begin{align*}
 \Psi _B(f) &= \Psi _B(f(1+|f|^2)^{-1}(1+|f|^2)) = \Psi _B(f(1+|f|^2)^{-1})\, \Psi _B(1+|f|^2)\\
 &= \Phi(f(1+|f|^2)^{-1})\, \Phi(1+|f|^2) = \Phi(f(1+|f|^2)^{-1}(1+|f|^2))\\
 &= \Phi(f).
\end{align*}
\end{proof}

\begin{thm}[Borel function calculus]
\label{thm:borel-fun-calc}
\textsl{
Let $\fM$ be a Murray-von Neumann algebra and let $A$ be a normal element in $\fM$. Then there is a unique $\sigma$-normal $*$-homomorphism $f \mapsto f(A): \mathscr{B}_u \big( \mathrm{sp} (A) \big) \to \fM$ (called the Borel function calculus) mapping the constant function $1$ to $I$ and the identity transformation $\iota$ on $\mathrm{sp}(A)$ onto $A$. Furthermore, we have the following:
\begin{itemize}
    \item[(i)] Let $\fA$ be the smallest Murray-von Neumann subalgebra of $\fM$ containing $A$. Then for every Borel function $f$ on $\mathrm{sp}(A)$, $f(A)$ is in $\fA$.
    \item[(ii)] If $A$ is bounded, the restriction of the Borel function calculus to $C(\mathrm{sp}(A))$ gives the continuous function calculus. 
\end{itemize}
}
\end{thm}
\begin{proof}
Let $\fA$ be the smallest Murray-von Neumann subalgebra of $\fM$ containing $A$. From Remark \ref{rmrk:abelian_mvn}, we have a locally compact Hausdorff space $X$ with a positive Radon measure $\mu$ such that $\fA \cong L^0(X; \mu)$. By Proposition \ref{prop:borel-fun-calc}, there is a $\sigma$-normal $*$-homomorphism $\mathscr{B}_u \big( \mathrm{sp} (A) \big) \to L^0(X; \mu)$ mapping the constant function $1$ to $I$ and the identity transformation $\iota$ on $\mathrm{sp}(A)$ onto $A$. Bearing in mind that the bounded operators $(I \afrsum A^* \afrprod A)^{-1}, A \afrprod (I \afrsum A^* \afrprod A)^{-1}$ lie in $\fA _{\textrm{nb}} \cong L^{\infty}(X; \mu)$ (see Proposition \ref{prop:spec_inv}) which is norm-closed and the fact that $\fA$ is monotone complete, we may follow the argument used in Proposition \ref{prop:borel-fun-calc} to conclude the uniqueness of the mapping implementing the Borel function calculus. This also shows that $f(A) \in \fA$ for all Borel functions $f$ on $\mathrm{sp}(A)$. Since the Borel function calculus restricts to a norm-decreasing map from the $C^*$-algebra $\mathscr{B}_b(\mathrm{sp}(A))$ to $\fA _{\mathrm{nb}}$, (ii) follows from norm-approximation using the polynomial function calculus.
\end{proof}

For a Borel function $f \in \mathscr{B}_u (\C)$ and a normal element $A$ of a Murray-von Neumann algebra, we define $f(A)$ as $\restr{f}{\mathrm{sp}(A)}(A)$. Note that if $f$ is a Borel function on $S$, a Borel subset of $\C$, then we may extend $f$ to the whole of $\C$ by defining it to be $0$ on $\C \backslash S$. Let $\C [z, w]$ denote the polynomial ring in two variables over the field of complex numbers.

\begin{cor}[Polynomial approximation]
\label{cor:complex-poly-approx}
\textsl{
Let $\fM$ be a Murray-von Neumann algebra and let $A$ be a normal element in $\fM$. For a Borel function $f$ on $\C$, there is a net of two-variable complex polynomials $\{ p_i \}$ in $\C[z, w]$ such that the net $\{ p_i(A, A^*) \}$ in $\fM$ converges in measure to $f(A)$.
}
\end{cor}
\begin{proof}
Let $\fA$ be the Murray-von Neumann subalgebra of $\fM$ given by the $\mathfrak{m}$-closure of the $*$-subalgebra of $\fM$ consisting of two-variable complex polynomials in $A, A^*$. Clearly $\fA$ is the smallest Murray-von Neumann subalgebra of $\fM$ which contains $A$. By Theorem \ref{thm:borel-fun-calc}, $f(A) \in \fA$ and the assertion follows.
\end{proof}

\begin{cor}
\textsl{
Let $A$ be a normal element in a Murray-von Neumann algebra $\fM$ and $f$ be a Borel function on $\C$.
For a self-adjoint element $B \in \fM$, if $A \afrprod B = B \afrprod A$, then $f(A) \afrprod B  = B \afrprod f(A)$. In particular, for a projection $E \in \fM$, if $A \afrprod E = E \afrprod A$, then $f(A \afrprod E)  = f(A) \afrprod E = E \afrprod f(A)$.
}
\end{cor}
\begin{proof}
If $B$ is self-adjoint and $A \afrprod B = B \afrprod A$, then $A^* \afrprod B = B \afrprod A^*$. Since the assertion holds for two-variable complex polynomials in $A, A^*$, the conclusion follows from Corollary \ref{cor:complex-poly-approx} by an approximation argument in the $\mathfrak{m}$-topology.
\end{proof}

\begin{cor}
\label{cor:real-poly-approx}
\textsl{
Let $A_1, A_2$ be self-adjoint elements in a Murray-von Neumann algebra $\fM$,  and $f$ be a Borel function on $\C$. Let $B_1, B_2 \in \fM$. If $A_1 ^n \afrprod B_1 = B_2 \afrprod A _2 ^n$ for all $n\in \N$, then $f(A_1) \afrprod B_1 = B_2 \afrprod f(A_2)$.
}
\end{cor}
\begin{proof}
From the hypothesis, for a polynomial function $p$ on $\R$, we have $p(A_1) \afrprod B_1 = B_2 \afrprod p(A_2)$. The assertion follows from Corollary \ref{cor:complex-poly-approx} by an approximation argument in the $\mathfrak{m}$-topology.
\end{proof}

\begin{cor}
\textsl{
Let $\fM _1$ and $\fM _2$ be Murray-von Neumann algebras and $\Phi : \fM _1 \to \fM _2$ be a morphism. Let $A$ be a normal element of $\fM _1$ and $f$ be a Borel function on $\mathrm{sp}(A)$. Then $\Phi(A)$ is a normal element of $\fM _2$, $\mathrm{sp}(\Phi(A)) \subseteq \mathrm{sp}(A)$, and $\Phi(f(A)) = f(\Phi(A))$.
}
\end{cor}
\begin{proof}
Since $A^* \afrprod A = A \afrprod A^*$, we have $\Phi(A^* \afrprod A) = \Phi(A \afrprod A^*)$ which implies that $\Phi(A)^* \afrprod \Phi(A) = \Phi(A) \afrprod \Phi(A)^*$. Hence $\Phi(A)$ is a normal element of $\fM _2$. For $\lambda \in \C$, if $A - \lambda I$ has an inverse $B$ in $(\fM _1)_{\textrm{nb}}$, then $\Phi(B) \in (\fM _2)_{\textrm{nb}}$ and $(\Phi(A) - \lambda I) \afrprod \Phi(B) = I$. Thus $\mathrm{sp}(\Phi(A)) \subseteq \mathrm{sp}(A)$. As $\Phi$ is a $*$-homomorphism, for any polynomial $p$ in $\C[z, w]$, we have $\Phi(p(A, A^*)) = p(\Phi(A), \Phi(A)^*)$. As $\Phi$ is $\mathfrak{m}$-continuous, the assertion follows from Corollary \ref{cor:complex-poly-approx}.
\end{proof}

\begin{cor}
\label{cor:comp_borel}
\textsl{
Let $A$ be a normal element in a Murray-von Neumann algebra $\fM$ and $f, g$ be Borel functions on $\C$. Then $$(f \circ g)(A) = f(g(A)).$$
}
\end{cor}
\begin{proof}
Let $g$ be a fixed Borel function on $\C$. Note that the mapping $f \mapsto f \circ g : \mathscr{B}_u(\C) \to \mathscr{B}_u(\C)$ is a $\sigma$-normal $*$-homomorphism. Hence the two mappings $f \mapsto (f \circ g)(A) : \mathscr{B}_u(\C) \to \fM$ and $f \mapsto f(g(A)) : \mathscr{B}_u(\C) \to \fM$ are $\sigma$-normal $*$-homomorphisms from $\mathscr{B}_u(\C)$ to $\fM$ taking the constant function $1$ onto $I$ and the identity transformation $\iota$ onto $g(A)$. From the uniqueness of the Borel function calculus of $g(A)$ (see Theorem \ref{thm:borel-fun-calc}), we conclude that $(f \circ g)(A) = f(g(A))$ for all Borel functions $f$ on $\C$.
\end{proof}

\begin{cor}
\label{cor:unique_square_root}
\textsl{
Let $A$ be a positive element in a Murray-von Neumann algebra $\fM$. Then there is a unique positive element $H$ in $\fM$ such that $A = H^2$. ($H$, also denoted by $\sqrt{A}$, is called the positive square root of $A$.)
}
\end{cor}
\begin{proof}
Note that $\R _{+}$ is a Borel subset of $\C$. Let $f: \R_{+} \to \R_{+}$ denote the Borel function on $\R _{+}$ given by $t \mapsto t^2$, and $g : \R_ {+} \to \R _{+}$ denote the Borel function on $\R _{+}$ given by $t \mapsto \sqrt{t}$. Note that $f \circ g = g \circ f = \iota$ on $\R _{+}$. Thus from Corollary \ref{cor:comp_borel}, we have $A = (f \circ g)(A) = f(g(A)) = (\sqrt{A})^2$. Let $B$ be another positive element in $\fM$ such that $A = f(B) = B^2$. From Corollary \ref{cor:comp_borel}, we observe that $B = (g \circ f)(B) = g(f(B)) = g(B^2) = g(A) = \sqrt{A}$.
\end{proof}

\begin{prop}
\label{prop:pos_desc_mvn}
\textsl{
Let $\fM$ be a Murray-von Neumann algebra and $A \in \fM$. Then the following are equivalent:
\begin{itemize}
    \item[(i)] $A$ is positive;
    \item[(ii)] There is a self-adjoint element $H$ in $\fM$ such that $A = H^2$;
    \item[(iii)] There is an element $B$ in $\fM$ such that $A = B^*B$.
\end{itemize}
}
\end{prop}
\begin{proof}
\noindent (i) $\Longrightarrow$ (ii). Choose $H = \sqrt{A}$ using the Borel function calculus.
\vskip 0.1in
\noindent (ii) $\Longrightarrow$ (iii). Choose $B = H$. 
\vskip 0.1in
\noindent (iii) $\Longrightarrow$ (i). Let $\{ B_i \}$ be a net in $\fM _{\textrm{nb}}$ converging in measure to $B$. Then $\{ B_i ^* B_i \}$ is a net of positive operators in $\fM _{\textrm{nb}}$ converging in measure to $B ^* B$. By Proposition \ref{prop:mtilde_pos}, we conclude that $B^*B$ is positive.
\end{proof}

\begin{cor}
\label{cor:pos_conj}
\textsl{
Let $\fM$ be a Murray-von Neumann algebra and $A, B$ be self-adjoint elements in $\fM$ such that $A \le B$. Then for every element $C$ in $\fM$, we have $$C^*AC \le C^*BC.$$
}
\end{cor}
\begin{proof}
Since $B - A$ is positive, by Proposition \ref{prop:pos_desc_mvn}, there is a self-adjoint element $H$ in $\fM$ such that $B - A = H^2$. We have $C^*(B-A)C = (HC)^*(HC)$ and again from Proposition \ref{prop:pos_desc_mvn}, we conclude that $C^*(B-A)C = C^*BA - C^*AC$ is positive.
\end{proof}

In the rest of this section, we use the Borel function calculus and approximation techniques in the $\mathfrak{m}$-topology to transfer many standard operator inequalities for bounded self-adjoint operators to the setting of self-adjoint elements in Murray-von Neumann algebras. Before we dive into the results, we first recall the definition of operator monotone and operator convex functions.

\begin{definition}
Let $f$ be a real-valued continuous function defined on the interval $\Gamma \subseteq \R$. The function $f$ is said to be {\it operator monotone} if for every pair of bounded self-adjoint operators $A, B$ acting on a Hilbert space with spectra in $\Gamma$ and such that $A \le B$, we have $$f(A) \le f(B).$$ The function $f$ is said to be {\it operator convex} if for every pair of bounded self-adjoint operators $A, B$ acting on a Hilbert space with spectra in $\Gamma$, we have $$f(t A + (1-t)B) \le t f(A) + (1-t)f(B).$$ 
\end{definition}
Note that all operator monotone and operator convex functions are continuous.

\begin{lem}
\label{lem:trace-approx}
\textsl{
Let $A$ be a positive operator in a $C^*$-algebra $\mathfrak{A}$, $E$ be a projection in $\mathfrak{A}$ and $\tau$ be a norm-continuous trace functional on $\mathfrak{A}$. Then for any real-valued continuous function $f$ on $[0, \infty)$, we have $$\tau \big( f(\sqrt{A} E \sqrt{A}) \big) = \tau \big( f(EAE) \big).$$
}
\end{lem}
\begin{proof}
Let $\lambda > 0$ be such that the spectrum of $\sqrt{A} E \sqrt{A}$ and $EAE$ lie in $[0, \lambda]$. Since $\tau$ is a trace functional, note that $\tau \big( (\sqrt{A} E \sqrt{A} )^n \big) = \tau \big( (AE)^n \big) = \tau \big ((EAE)^n \big)$ ($n \in \N$) and thus for any real polynomial $p$, we have
$$\tau \big (p(\sqrt{A} E \sqrt{A}) \big) = \tau \big( p(EAE) \big).$$
We may uniformly approximate $f$ on $[0,\lambda]$ by real polynomials to reach the desired conclusion.
\end{proof}

\begin{lem}
\textsl{
Let $\mathscr{R}$ be a finite von Neumann algebra. Let $A$ be a positive operator in $\mathscr{R}$ and $\{ E_n \}_{n \in \N}$ be an increasing sequence of projections converging to $I$ in the strong-operator topology. Let $f : [0, \infty) \to \R$ be an operator monotone function. Then $\{ f(\sqrt{A}E_n\sqrt{A}) \}_{n \in \N}$ is an increasing sequence of positive operators with least upper bound $f(A)$. 
}
\end{lem}
\begin{proof}
Without loss of generality, we may assume that $f(0) = 0$ and $f$ is positive-valued. The least upper bound of $\{ \sqrt{A} E_n \sqrt{A} \}_{n \in \N}$ is given by $A$. For $m \ge n$, we have $\sqrt{A} E_n \sqrt{A} \le \sqrt{A}E_m \sqrt{A} \le A$. Since $f$ is operator monotone, we observe that $\{ f(\sqrt{A}E_n\sqrt{A}) \}_{n \in \N}$ is an increasing sequence of positive operators bounded above by $f(A)$. Let $B \in \mathscr{R}^{+}$ be the least upper bound of $\{ f(\sqrt{A} E_n \sqrt{A}) \}$. Let $\tau$ be a normal tracial state on $\mathscr{R}$. By Lemma \ref{lem:trace-approx}, we have $$\tau(f(\sqrt{A}E_n \sqrt{A})) = \tau(f(E_n A E_n)).$$ By \cite[Theorem 2.5]{hansen-pedersen}, $f$ is operator concave and thus $$E_nf(A)E_n \le f(E_n AE_n).$$ It follows that $$\tau(E_nf(A)E_n) \le \tau(f(E_nAE_n)) = \tau(f(\sqrt{A}E_n\sqrt{A})) \le \tau(f(A)).$$
Taking the limit as $n \rightarrow \infty$ in the strong operator-topology, we observe that $ \tau(f(A) - B) = 0$. Since this is true for all normal tracial states $\tau$ and $B \le f(A)$, we conclude that $B = f(A)$. In other words, $f(A)$ is the least upper bound of $\{ f(\sqrt{A} E_n \sqrt{A}) \}_{n \in \N} $.
\end{proof}

\begin{thm}
\label{thm:op-mon}
\textsl{
Let $\mathscr{R}$ be a finite von Neumann algebra. Let $f: [0, \infty) \to \R$ be an operator monotone function and $A, B \in \afr$ be positive operators such that $A \le B$. Then $f(A) \le f(B)$.
}
\end{thm}

\begin{proof}
Without loss of generality, we may assume that $f(0) = 0$ and $f$ is positive-valued. To begin with, let us assume that $A$ is {\it bounded}, that is, $A \in \mathscr{R}$. For $n \in \N$, let $F_n$ be the spectral projection of $B$ corresponding to the interval $[0, n]$. Since $0 \le F_n \afrprod  A \afrprod F_n \le F_n \afrprod B \afrprod F_n $ and $F_n \afrprod  A \afrprod F_n$ and $F_n \afrprod  B \afrprod F_n$ are bounded positive operators, we have $$ F_n \afrprod  f(A) \afrprod  F_n \le f(F_n \afrprod A \afrprod F_n) \le f(F_n \afrprod B \afrprod F_n) = f(B) \afrprod F_n \le f(B).$$ For the first inequality, we used the fact that every positive-valued operator monotone function on $[0, \infty)$ is operator concave (see \cite[Theorem 2.5]{hansen-pedersen}). Taking the limit as $n \rightarrow \infty$, we conclude that $f(A) \le f(B)$.

Next we prove the general case where $A$ may be unbounded. For $n \in \N$, let $E_n$ be the spectral projection of $A$ corresponding to the interval $[0, n]$. Since $\sqrt{A} \afrprod E_n \afrprod \sqrt{A} \le A \le B$ and $\sqrt{A} \afrprod E_n \afrprod \sqrt{A}$ is bounded, by the previous result, we have $f(A) \afrprod E_n = f(\sqrt{A} \afrprod E_n \afrprod \sqrt{A}) \le f(B)$. Taking the limit as $n \rightarrow \infty$, we conclude that $f(A) \le f(B)$.
\end{proof}

\begin{cor}
\textsl{
Let $\mathscr{R}$ be a finite von Neumann algebra. Let $0 \le A \le B$ be positive operators in $\afr$. Then
\begin{itemize}
    \item[(i)] $\sqrt{A} \le \sqrt{B}$;
    \item[(ii)] $\log (I \afrsum A) \le \log (I \afrsum B)$.
\end{itemize}
}
\end{cor}
\begin{proof}
As the mappings $t \mapsto \sqrt{t}, t \mapsto \log(1+t) : [0, \infty) \to \R$ are operator monotone, the assertion follows.
\end{proof}

\begin{thm}
\label{thm:op-conv}
\textsl{
Let $\mathscr{R}$ be a finite von Neumann algebra. For an operator convex function $f: [0, \infty) \to \R$ with $f(0) = 0$, we have the following:
\begin{itemize}
    \item[(i)] $f(tA \afrsum (1-t)B) \le tf(A) \afrsum (1-t)f(B)$ for all positive operators $A, B$ in $\afr$;
    \item[(ii)] $f(V^* \afrprod A \afrprod V) \le V^* \afrprod f(A) \afrprod V$ for every positive contraction $V$ in $\mathscr{R}$ (that is, $\|V\| \le 1$) and every positive operator $A$ in $\afr$;
    \item[(iii)] $f(V^* \afrprod A \afrprod V \afrsum  W^* \afrprod B \afrprod W) \le V^* \afrprod f(A) \afrprod V \afrsum  W^* \afrprod f(B) \afrprod W$ for all $V, W$ in $\mathscr{R}$ with $V^*V + W^*W \le I$ and all positive operators $A, B$ in $\afr$;
    \item[(iv)] $f(E \afrprod A \afrprod E) \le E \afrprod f(A) \afrprod E$ for every projection $E$ in $\mathscr{R}$ and every positive operator $A$ in $\afr$.
\end{itemize} 
}
\end{thm}
\begin{proof}
With Theorem \ref{thm:matrix-iso} at hand, the matrix techniques used in the proof of  \cite[Theorem 2.1]{hansen-pedersen} can be directly adapted to the setting of Murray-von Neumann algebras to prove that (iv) $\Rightarrow$ (i) $\Rightarrow$ (ii) $\Rightarrow$ (iii). Thus we need only prove (iv).

Let $E$ be a projection in $\mathscr{R}$. By \cite[Theorem 2.4]{hansen-pedersen}, there is an operator monotone function $g$ on $[0, \infty)$ such that $f(t) = tg(t)$. Since $g$ is operator monotone and $\sqrt{A} \afrprod E \afrprod \sqrt{A} \le A$, by Theorem \ref{thm:op-mon} we see that $g(\sqrt{A} \afrprod E \afrprod \sqrt{A}) \le g(A)$. Thus by Corollary \ref{cor:pos_conj}, we have
\begin{equation}
\label{eqn:op-conv1}
\big(E\afrprod \sqrt{A} \big) \afrprod  g(\sqrt{A} \afrprod E \afrprod \sqrt{A}) \afrprod \big(\sqrt{A} \afrprod E \big) \le \big( E \afrprod \sqrt{A} \big) \afrprod g(A) \afrprod  \big(\sqrt{A} \afrprod E \big).    
\end{equation}
Note that for $n \in \N$, we have the following identity, $$(\sqrt{A} \afrprod E \afrprod \sqrt{A})^n \afrprod (\sqrt{A} \afrprod E) = (\sqrt{A} \afrprod E) \afrprod (E \afrprod A \afrprod E)^n.$$ Hence from Corollary \ref{cor:real-poly-approx}, we observe that
\begin{align*}
g(\sqrt{A} \afrprod E \afrprod \sqrt{A})\afrprod (\sqrt{A} \afrprod E) &= (\sqrt{A} \afrprod E) \afrprod g(E \afrprod A \afrprod E),\\
\sqrt{A} \afrprod g(A) \afrprod \sqrt{A} &= A \afrprod g(A).
\end{align*}
Plugging the above identities in (\ref{eqn:op-conv1}), we have $$(E\afrprod A \afrprod E) \afrprod g(E \afrprod A \afrprod E) \le  E \afrprod (A \afrprod g(A) ) \afrprod E.$$ 
Equivalently, $$f(E \afrprod A \afrprod E) \le E \afrprod f(A) \afrprod E.$$
\end{proof}

\bibliographystyle{plain}

\bibliography{reference}

\end{document}